\documentclass[11pt,reqno]{amsart}
 
\newcommand{\mysection}[1]{\section{#1}
      \setcounter{equation}{0}}
\usepackage{color}

\newtheorem{theorem}{Theorem}[section]
\newtheorem{lemma}[theorem]{Lemma}

\newtheorem{corollary}[theorem]{Corollary}

\theoremstyle{definition}
\newtheorem{assumption}{Assumption}[section]

\newtheorem{example}{Example}[section]

\theoremstyle{remark}
\newtheorem{remark}{Remark}[section]

\newcommand\bR{\mathbb{R}}

\newcommand\frB{\mathfrak{B}}

\newcommand\frK{\mathfrak{K}}

\newcommand\cK{\mathcal{K}}
\newcommand\cL{\mathcal{L}}

\newcommand\cQ{\mathcal{Q}}

\makeatletter
 \newcommand{\sumstar}%
 {\operatornamewithlimits{\sum@\kern-.2em\raise1ex\hbox{*}}}
 \makeatother

\begin{document}

\title[Accelerated schemes]{
 Accelerated finite difference 
schemes for second order degenerate
elliptic and parabolic problems in the whole space}

\author[I. Gy\"ongy]{Istv\'an Gy\"ongy}
\address{School of Mathematics and Maxwell Institute,
University of Edinburgh,
King's  Buildings,
Edinburgh, EH9 3JZ, United Kingdom}
\email{gyongy@maths.ed.ac.uk}

\author[N.  Krylov]{Nicolai Krylov}%
\thanks{The work of the second author was partially supported
by NSF grant DMS-0653121}
\address{127 Vincent Hall, University of Minnesota,
Minneapolis,
       MN, 55455, USA}
\email{krylov@math.umn.edu}

\subjclass{65M15, 35J70, 35K65}
\keywords{Cauchy problem, 
finite differences, extrapolation to the limit, 
 Richardson's method }

\begin{abstract}
We give sufficient conditions under which the 
convergence of finite difference approximations 
in the space variable of possibly degenerate second 
order parabolic and elliptic equations can be accelerated 
to any given order of convergence by Richardson's method.  
\end{abstract}

\maketitle

\mysection{Introduction}                   \label{section02.04.06}

This is the third article of a series  
studying a class of finite difference equations,  
related to finite difference approximations in
{\em the space variable\/} 
of second order parabolic and elliptic PDEs in $\bR^d$. 
These PDEs are given on the whole $\bR^d$ in the space variable, 
and may degenerate and become first order PDEs. 
Denote by  $u_h$ the solutions of the finite 
difference equations
corresponding 
to a given grid with mesh-size $h$. 
By 
shifting  the grid 
so that $x$ becomes a grid point we 
define  $u_h$ for all $x\in\bR^d$ 
rather than only at the points of the original grid. 
In \cite{GK1} and \cite{GK2}, 
the first and second articles of the series,   
we focus on the smoothness in $x$ of
$u_h$, rather than their 
convergence for $h\to0$. The main results 
in \cite{GK1} and \cite{GK2} 
give estimates, independent of $h$, for the first order 
derivatives $Du_h$ and for derivatives 
$D^{k}u_h$ in $x$ of any order $k$, respectively. 

In the present paper one of our main concerns is  
the smoothness of the approximations $u_h$ in $(x,h)$.   
In particular, we are interested 
in the convergence of $u_h$,  and their {\em derivatives\/} 
in $x$,  
in the supremum norm, as $h\to0$.  
We give conditions ensuring 
that for any 
given integer $k\geq0$ the approximations $u_h$ admit 
power series expansions up to order $k+1$  
in $h$ near $0$
like
\begin{equation}
                                           \label{08.12.14.1}
u_{h}=\sum_{j=0}^{k}h^{j}u^{(j)}+h^{k+1}r_{h},
\end{equation} 
and such that the coefficients are bounded functions  
of $(t,x)\in[0,T]\times\bR^d$ for fixed  $T>0$ 
in the case of parabolic equations, 
and, with the exception of  $r_{h}$, 
are independent of $h$. This is Theorem   
\ref{theorem 6.25.08.08}, our first result on Taylor's formula 
for $u_h$ in $h$. 
We obtain it by proving first Theorems 
\ref{theorem 08.11.2.1} and 
\ref{theorem 08.11.7.2} 
 below on the solvability of the PDE that is being approximated, 
and of a system of degenerate 
parabolic PDEs, respectively,  for the coefficients   $u^{(j)}$,
$j=0,\dots, k$. 
Of course, $u^{(0)}$ is the true solution
of the corresponding PDE.
The remainder term $r_{h}$  satisfies a  
finite difference equation, with the same difference  
operator appearing in the equation for $u_h$, and we estimate 
$r_{h}$
by making use of the maximum principle enjoyed by this operator. 
This is a standard approach to get power series expansions 
for finite difference approximations in general, 
and it 
works well in many situations, when suitable results 
regarding the equations for the coefficients 
$u^{(j)}$ 
are
available.   In our situation it requires some facts either from the
theory  of diffusion processes or from the theory of degenerate 
parabolic equations. However, we do not use any facts 
from these theories. We prove 
Theorem \ref{theorem 08.11.2.1}, and hence 
Theorem \ref{theorem 08.11.7.2}, relying on results 
on finite difference schemes, obtained in 
\cite{GK2} by 
elementary techniques. 
It is worth saying that since long ago 
finite difference equations were already used
to prove the solvability of partial differential equations
(see, for instance, \cite{La73} and \cite{Li}). Our contribution
lies in considering  degenerate equations.

After establishing the expansions of $u_{h}$ in $h$
not only we obtain the possibility to prove the convergence
of $u_{h}$ to the true solution in the sup norm  as $h\to0$ but also
the possibility to accelerate it to any order
under appropriate assumptions.  
We prove the latter
  by taking linear combinations of finite 
difference approximations corresponding to different 
mesh-sizes. 
 This method is especially effective 
when many of the coefficients in the expansion of 
$u_h$ are zero. 
These results are given by Theorem 
\ref{11.25.11.08} and Corollary \ref{corollary 10.25.11.08}. 
Their counterparts in the elliptic case are presented 
by Corollary \ref{corollary 3.08.02.08}.  

The idea of accelerating the 
convergence of finite difference approximations  
 in the above way  is well-known in numerical analysis.  
It is due to L.F.~Richardson, who showed that it works  
in some cases and demonstrated its usefulness in 
\cite{R} and \cite{RG}. This method is often called 
{\em Richardson's method\/} or {\em extrapolation to the 
limit\/}, and is applied to various types of approximations. 
The reader is referred to the survey papers 
\cite{Br} and \cite{J} for a review  
on the history of the method and on the scope of its 
applicability and to the textbooks 
(for instance,  \cite{Ma} and \cite{MS})
concerning finite difference 
methods and their accelerations.

 We are interested  in approximating 
 in the sup norm  not only
the true solution but also its derivatives.
  Note  that even if the coefficients 
$u^{(j)}$ 
are bounded smooth functions of $(t,x)$, 
the derivatives 
$D^{k}u_{h}$ of $u_h$ in $x$ 
may not admit similar expansions, 
since the derivatives of $r_{h}$ 
  may not be bounded in $h$ near $0$. 
 Note also  that the bounds on 
the sups of $u^{(j)}$ and $r_{h}$ generally depend on $T$,
and  may  grow   exponentially in $T$.
This becomes a big obstacle on the way of
extending our results to the elliptic case. 

Our next result on power series 
expansions, Theorem \ref{theorem 1.25.11.08}, 
improves the previous theorem in two directions. 
It gives 
sufficient conditions such that for 
any given integer $k\geq0$

(a) $D^ku_h$  admits an expansion 
similar to \eqref{08.12.14.1},

(b) the bounds on the coefficients are independent of $T$.  

Having (a) we can approximate the $k$-th derivatives
of the true solution by $D^{k}u_{h}$ with   rate of order $h$
and accelerate the rate
under appropriate assumptions. We can also approximate 
the $k$-th derivatives of the true solution
with finite difference operators in place of $D^{k}$
applied to $u_{h}$, which is more convenient
in applications because it does not require computing
the derivatives of $u_{h}$.

We ensure (a) and (b) by relying heavily on 
derivative estimates, independent of $T$, 
obtained in \cite{GK1} and \cite{GK2} for 
solutions of finite difference equations.  
Property (b) of the expansions allows us to extend 
Theorem \ref{theorem 1.25.11.08} to the elliptic case. 
This extension is  
Theorem \ref{theorem 4.20.11.08}. 

As a consequence of the derivative estimates 
proved in \cite{GK2} we obtain also, see  
Theorem \ref{theorem 17.25.06.07} below,  
estimates, independent 
of $h$ and $T$, for the derivatives of $u_h$ in $x$ {\em and\/} $h$. 
Clearly, Theorem \ref{theorem 17.25.06.07} 
immediately implies Taylor's formula 
for $u_h$ 
in $h$, up to appropriate  
order, with bounded coefficients. It is interesting to 
notice that the converse implication 
does not hold: If for $k\geq1$ the 
function $u_h$ admits a
power series expansion up to order 
$k+1$ in $h$ near 0 with bounded coefficients, 
it does not imply, in general, that the derivative of 
$u_{h}$ in $h$ up to order $k+1$ are bounded functions. 
That is why Theorem \ref{theorem 1.25.11.08}
does not imply Theorem \ref{theorem 17.25.06.07}, and the latter
implies the former only if condition (i) 
in Theorem \ref{theorem 1.25.11.08} is satisfied. 
 Additional 
information on the behaviour of the derivatives of $u_h$ 
in $x$ and $h$ when $h$ is near $0$ is given by 
Theorem \ref{theorem 08.2.19.1}. The corresponding result 
in the elliptic case is Theorem~\ref{theorem 09.24.08.08}.

In this article we  are working with equations in the whole space
having in mind considering equations in bounded smooth domains
in a subsequent article. Still it may be worth noting
that the results of this article
 are applicable to the one dimensional ODE
$$
(1-x^{2})^{2}u''(x)-c(x)u(x)=f(x),\quad x\in(-1,1).
$$
The point is that one need not prescribe any boundary value
of $u$ at the points $\pm1$ and if one considers this equation
on all of $\bR$, the values of
its coefficients and $f$ outside $(-1,1)$ do not affect
the values of $u(x)$ for $|x|<1$.

\mysection{Formulation of the main results
for parabolic equations}

We fix  some numbers $h_{0},T\in(0,\infty)$ and
for  each number $h\in(0,h_{0}]$ 
we consider the integral equation 
\begin{equation}                              \label{equation}
u(t,x)=g (x)
+\int_0^t\big(L_hu(s,x)+f (s,x)\big)\,ds, 
\quad (t,x)\in H_T
\end{equation}
for $u$, where  $ g (x)$ and $ 
f (s,x)$  are
given  real-valued Borel functions of 
$x\in\bR^d$ and  
$(s,x)\in H_T=[0,T]\times\bR^d$,  
respectively, and 
 $L_h$ is a linear operator 
defined by 
\begin{equation}                          \label{1.24.11.08}
L_h\varphi(t,x)=L_h^0\varphi(t,x)
-c(t,x)\varphi(x), 
\end{equation}
\begin{equation}                          \label{08.12.16.1}
L_h^0\varphi(t,x)=\frac{1}{h}
\sum_{\lambda\in\Lambda_{1}}
q_{\lambda}(t,x )\delta_{h,\lambda}\varphi(x) 
+\sum_{\lambda\in\Lambda_{1}} 
p_{\lambda}(t,x )\delta_{h,\lambda}\varphi(x) , 
\end{equation}
for functions $\varphi$ on $\bR^d$. Here 
$\Lambda_{1}$ 
is a finite  subset of  $\bR^d$ 
such that
$0\not\in\Lambda_{1}$,    
$$
\delta_{h,\lambda}\varphi(x )=
\frac{1}{h}(\varphi(x +h\lambda)-\varphi(x )),
\quad\lambda\in\Lambda_{1}, 
$$  
$q_{\lambda}(t,x)\geq0$, $p_{\lambda}(t,x )$,
 and $c(t,x)$  are given 
real-valued Borel functions  of $(t,x)\in H_{\infty}
=[0,\infty)\times\bR^d$
for each 
$\lambda\in\Lambda_{1}$. 
Set $|\Lambda_1|^2=\sum_{\lambda\in\Lambda_1}|\lambda|^2$. 

As usual, we denote
$$
D^{\alpha}=D_{1}^{\alpha_{1}}...D_{d}^{\alpha_{d}},
\quad D_{i}=
\frac{\partial }{\partial x_i},\quad
|\alpha| =\sum_{i}\alpha_i, 
\quad
 D_{ij}=D_iD_j 
$$
for multi-indices $\alpha=(\alpha_1,\dots\alpha_d)$,
$\alpha_{i}\in\{0,1,\dots\}$.  
For smooth $\varphi$  
and integers $k\geq0$ we introduce
$D^{k}\varphi$ as the collection of partial derivatives
of $\varphi$ of order $k$, and define 
$$
|D^{k}\varphi|^{2}=\sum_{|\alpha|=k}|D^{\alpha}
\varphi|^{2},\quad
[\varphi]_{k}
=\sup_{x\in\bR^d}
|D^{k}\varphi(x)|,\quad |\varphi|_{k}=
\sum_{i\leq k}[\varphi]_{i}.
$$

 For functions $\psi_h$ depending on 
$h\in(0, h_0 ]$ 
the notation 
 $
D_h^k\psi_h 
 $  
 means the $k$-th 
derivative of $\psi$ in $h$. 
For Borel measurable bounded functions   
$\psi=\psi (t,x)$ on $  H_T$ 
we write   
$\psi\in \frB^{ m}=\frB^{ m}_{T}$  if, for each $t\in[0,T]$, 
$\psi (t,x)$ is continuous in $ \bR^{d}$
and for
all  multi-indices $\alpha$ with
 $|\alpha|\leq m$  
the generalized functions
$D^{\alpha} \psi (t,x)$  
are  
bounded on  $  H_T$.
In this case we use the notation 
$$
\|\psi\|^{2}_{m}= \sup_{H_{T}} 
\sum_{ |\alpha|\leq m}
| D^{\alpha}\psi (t,x)|^{2}. 
$$
This notation will be also used for functions $\psi$
independent of $t$.

Let $m\geq0$ be a fixed integer.
We make the following assumptions.

\begin{assumption}              \label{assumption 16.12.07.06}
For any $\lambda\in\Lambda_1$, 
 we have  $p_{\lambda},q_{\lambda},c,f , g\in\frB^{m}$
and, for $k=0,...,m$ and some  constants  $M_k$ 
we have
\begin{equation}
                                         \label{3.5.1}
\sup_{H_{T}} \big(
\sum_{\lambda\in\Lambda_{1}}
(|D^{k}q_{\lambda}|^{2}+|D^{k}p_{\lambda}|^{2}\big)
+|D^{k}c|^{2}\big)
\leq M^{2}_{k}.
\end{equation} 

\end{assumption}

 \begin{remark}                 \label{remark 9.19.01.08}
By Theorem 2.3 of  
\cite{GK1} under Assumption \ref{assumption 16.12.07.06}
for each $h\in(0,h_{0}]$,
there exists a unique bounded solution $u_{h}$
of \eqref{equation}, this solution  is
continuous in $H_{T}$, and
 all its derivatives in $x$ up to order
$m$ are bounded. Actually, in Theorem 2.3 of  
\cite{GK1} it is required that the derivatives
of the data up to order $m$ be continuous in $H_{T}$,
but its proof can be easily adjusted to 
include our case (see Remark \ref{remark 08.2.19.1} below).
 
\end{remark}
 
Naturally, we view \eqref{equation} as a
finite difference schemes for the problem  

\begin{equation}                               \label{08.11.2.1}
\frac{\partial}{\partial t}u(t,x)=
\cL u(t,x)+f (t,x), \quad t\in(0,T],\,x\in\bR^d,
\end{equation}
\begin{equation}                               \label{08.11.2.2}
u(0,x)=g (x),\quad x\in\bR^d,
\end{equation}
where 
\begin{equation}                               \label{08.11.3.6}
\cL:=\tfrac{1}{2}
\sum_{\lambda\in\Lambda_1}
\sum_{i,j=1}^dq_{\lambda}\lambda_i\lambda_jD_iD_j
+\sum_{\lambda\in\Lambda_1}
\sum_{i=1}^dp_{\lambda}\lambda_iD_i-c.
\end{equation}

By a solution of \eqref{08.11.2.1}-\eqref{08.11.2.2}
we mean a  bounded continuous 
 function $u(t,x)$ on $H_{T}$, such that 
it belongs to $\frB^{2}$
and satisfies
\begin{equation}                               \label{08.11.2.3}
u(t,x)=g (x)+\int_{0}^{t}[\cL u(s,x)+
f (s,x)]\,ds
\end{equation}
in $H_{T}$ in the sense of generalized functions, that is,
for any $t\in[0,T]$ and $\phi\in C^{\infty}_{0}(\bR^{d})$
$$
\int_{\bR^{d}}\phi(x)u(t,x)\,dx=
\int_{\bR^{d}}\phi(x)g (x)\,dx
+\int_{0}^{t}\int_{\bR^{d}}\phi(-cu+f)(s,x)\,dxds
$$
\begin{equation}                               \label{08.11.7.1}
+\int_{0}^{t}\int_{\bR^{d}}\phi
\sum_{\lambda\in\Lambda_1}\big(\tfrac{1}{2}
\sum_{i,j=1}^d 
q_{\lambda} \lambda_{i}\lambda_{j}D_{i}D_{j}u+ \sum_{i=1}^d
p_{\lambda}\lambda_{i}D_{i}u\big)(s,x)\,dxds.
\end{equation}
Observe that if $u\in\frB^{2}$, then \eqref{08.11.7.1}
implies that \eqref{08.11.2.3} holds almost everywhere
with respect to $x$ and if $u\in\frB^{3}$
then the second derivatives of $u$ in $x$ are continuous in $x$
and \eqref{08.11.2.3} holds everywhere.

The reader can find in \cite{Kr08} a discussion showing that
in all practically interesting cases of parabolic
equations like \eqref{08.11.2.3} the operator $\cL$ can be represented
as in \eqref{08.11.3.6}, so that considering operators $L_{h}^{0}$
in form \eqref{08.12.16.1} is rather realistic.

The 
following theorem on
existence and uniqueness of solutions
is a classical result (see, for instance, 
\cite{Ol1}, \cite{Ol}, \cite{OR}) which we are going to obtain
by using finite-difference approximations.

\begin{theorem}                     \label{theorem 08.11.2.1}
Let Assumption \ref{assumption 16.12.07.06} 
hold with $m\geq2$. Then equation 
\eqref{08.11.2.3} has a unique solution 
$u^{(0)}\in\frB^2=\frB^2_{T}$. 
Moreover, $u^{(0)}\in\frB^m_{T}$ and 
\begin{equation}
                                     \label{08.11.3.2}
\|u^{(0)}\|_m\leq N(\|f \|_m+\|g \|_m), 
\end{equation} 
where $N$ is a constant, depending only on 
$d$, $m$,   $ | \Lambda_1 |  $,   
$M_{ 0}$,\dots, $M_m$, and
 $T$.
\end{theorem}

Observe that this result is rather sharp in what concerns the smoothness
of solutions, which is seen if all the coefficients of $\cL$
are identically zero and $f$ is independent of $t$ in which case
the solution is $tf(x)+g(x)$.

The existence part in Theorem \ref{theorem 08.11.2.1} is proved
in Section \ref{section 1.25.11.08}
 and uniqueness in Section \ref{section 08.12.17.1}. 

In Section \ref{section 1.25.11.08} a 
repeated application  of this theorem allows us
to prove  
 a result on the solvability of   
\eqref{08.11.2.6} below.
First introduce 
\begin{equation}                        \label{08.11.2.5}
\cL^{(i)}:=
\tfrac{1}{(i+1)(i+2)}\sum_{\lambda\in\Lambda_1}
q_{\lambda}\partial_{\lambda}^{i+2}
+\tfrac{1}{i+1}\sum_{\lambda\in\Lambda_1}
p_{\lambda}\partial_{\lambda}^{i+1}, 
\end{equation}
  where 
\begin{equation}                             \label{08.12.16.2}
\partial_{\lambda}\varphi
:=\sum_{i}\lambda_iD_i\varphi
\end{equation}
is the derivative of $\varphi$ 
in the direction of $\lambda$. 
Consider the system of equations
\begin{equation}                             \label{08.11.2.6}
u^{(j)}(t,x)=
\int_{0}^{t}\big(\cL u^{(j)}(s,x)
+\sum_{i=1}^{j}C^{i}_{j}\cL^{(i)}u^{(j-i)}(s,x)
\big)\,ds,
\end{equation}
$(t,x)\in H_T$, $j=1,\dots,k$.
\begin{remark}
                                       \label{remark 08.11.6.1}
Quite often in the article we use the following 
symmetry condition: 
\smallskip
 
 (S) $\Lambda_{1}=-\Lambda_{1}$ and
$q_{\lambda}=q_{-\lambda}$ 
for all $\lambda\in\Lambda_1$.
\smallskip

\noindent
Notice that, if condition (S) holds, then 
$$
h^{-1}\sum_{\lambda\in\Lambda_{1}}q_{\lambda}(t,x)
\delta_{h,\lambda}\varphi(x)
=(1/2)\sum_{\lambda\in\Lambda_{1}}q_{\lambda}(t,x)
\Delta_{h,\lambda}\varphi(x),
$$
where
$$
\Delta_{h,\lambda}\varphi(x)
=h^{-2}
(\varphi(x+h\lambda)-2\varphi(x)+\varphi(x-h\lambda)). 
$$
\end{remark}

\begin{theorem}                     \label{theorem 08.11.7.2}
Let $k\geq1$ be an integer.
(i) If Assumption \ref{assumption 16.12.07.06} is 
satisfied with $m\geq 3k+2$, 
then \eqref{08.11.2.6} has 
a unique solution $\{u^{(j)}\}_{j=1}^k$, such that
\begin{equation}                             \label{08.11.7.3}
u^{(j)}\in \frB^{m-3j} ,\quad
\| u^{(j)}\|_{m-3j}\leq N(\|f\|_m+\|g\|_m)
\end{equation}
for $j=1,\dots,k$. 

(ii) If the symmetry condition (S) holds 
and Assumption
\ref{assumption 16.12.07.06} is 
satisfied with $m\geq 2k+2$, then 
\eqref{08.11.2.6} has 
a unique solution $\{u^{(j)}\}_{j=1}^k$, such that  
\begin{equation}                          \label{08.11.7.4} 
u^{(j)}\in \frB^{m-2j} , 
\quad 
\| u^{(j)}\|_{m-2j}\leq N(\|f\|_m+\|g\|_m)
\end{equation} 
for $j=1\dots,k$. 
In addition, if
\begin{equation}                          \label{08.11.7.5}
p_{-\lambda}=-p_{\lambda}, 
\quad \text{for $\lambda\in\Lambda_1$}, 
\end{equation} 
then  
\begin{equation}                          \label{08.11.7.7}
u^{(j)}=0,
\end{equation}
for odd numbers $j\leq k$. 

In all cases the constants $N$ depends only on $d$, $m$,   
$ | \Lambda_1 | $,  
$M_{ 0} , \dots,M_m$, and $T$.
\end{theorem}

The next series of results is related to 
the possibility of expansion
\begin{equation}                       \label{08.11.2.9}
u_h(t,x)=u^{(0)}(t,x)
+\sum_{1\leq j\leq k}\frac{h^j}{j!}u^{(j)}(t,x)
+ h^{k+1} r_h(t,x), 
\end{equation}
for all $(t,x)\in H_T$ and $h\in(0,h_{0}]$, where 
$u_{h}$ is 
 the  
unique bounded solution 
of \eqref{equation} 
(see Remark \ref{remark 9.19.01.08}) and $r_h$ 
is a function on $H_T$ 
defined for each $h\in(0,h_0]$ such that 
\begin{equation}                       \label{08.11.2.10}
|r_h(t,x)|
\leq N (\|f\|_m+\|g\|_m) 
\end{equation}
for all $(t,x)\in H_T$, $h\in(0,h_{0}]$.

Introduce
$$
\chi_{h,\lambda}=q_{\lambda}+hp_{\lambda}.
$$
\begin{assumption}  
                            \label{assumption 1.26.11.06}   
For all $(t,x)\in H_T$, $h\in(0,h_{0}]$,
and $\lambda\in\Lambda_{1}$,
\begin{equation}                     \label{1.10.02.08}
\chi_{h,\lambda}(t,x )\geq 0.
\end{equation}

\end{assumption}

\begin{assumption}            \label{assumption 14.23.01.08} 
We have 
\begin{equation*}
                                         \label{1.01.02.08}
\sum_{\lambda\in\Lambda_{1}}\lambda q_{\lambda}(t,x )=0
\quad \text{for all $(t,x)\in H_T$}.
\end{equation*}

\end{assumption}
Notice that condition (S) is stronger
than Assumption \ref{assumption 14.23.01.08}.

\begin{theorem}                    \label{theorem 6.25.08.08} 
Let Assumption  \ref{assumption 16.12.07.06} with $m\geq3$
and Assumption \ref{assumption 1.26.11.06} 
hold. 
Let $k\geq0$ be an integer. 
Then 
expansion \eqref{08.11.2.9}  
holds with
$r_{h}$ satisfying 
 \eqref{08.11.2.10}, 
provided one of the following 
  conditions 
is met:

(i) $m\geq3k+3$ and Assumption \ref{assumption 14.23.01.08} 
holds;
    
(ii) $m\geq2k+3$ and condition (S) holds;

(iii) $k$ is odd, $m\geq2k+2$, and conditions (S) 
and \eqref{08.11.7.5} are satisfied. 

\noindent
In each of the cases (i)-(iii) the constant  
  $N$  depends only   on $d$, $m$,   
$ | \Lambda_1 | $, 
  $M_{ 0},\dots,M_m$, and $T$. 
 In case (iii) we have $u^{(j)}=0$ 
for all odd $j$ in expansion 
\eqref{08.11.2.9}.
\end{theorem}

We prove this theorem in 
Section \ref{section 5.23.11.08}. 
  The following corollary is one of  
the results of \cite{DK}
proved there by using the theory of diffusion processes. 
We obtain it immediately from case (iii) with $k=1$. 
Of course, the result is well known
for uniformly nondegenerate equations but we do not assume
any nondegeneracy of $\cL$, which becomes just a zero operator at those
points where $q_{\lambda}=p_{\lambda}=c=0$.
 
\begin{corollary}
                                         \label{corollary 09.12.17.1}

Let conditions (S) and \eqref{08.11.7.5}
be satisfied.  Let Assumption  
\ref{assumption 16.12.07.06} with $m=4$
and Assumption \ref{assumption 1.26.11.06} 
 hold. Then  we have $|u_{h}-u_{0}|\leq Nh^{2}$. 
\end{corollary}

Actually, in \cite{DK} a full discretization in time and space
is considered for parabolic equations, so that, formally,
Corollary \ref{corollary 09.12.17.1} does not yield the
corresponding result of  \cite{DK}. On the other hand,
a similar corollary can be derived from
Theorem \ref{theorem 4.20.11.08} below which treats
elliptic equations and it does imply the corresponding result
of  \cite{DK}. It also generalizes it because in 
\cite{DK} one of the
assumptions,  unavoidable for the methods used there, 
is that 
$ q_{\lambda} =r_{\lambda}^{2}$ with 
functions $r_{\lambda}$
that have four bounded
derivatives in
$x$, which may easily be not the 
case under the assumptions of 
Theorem \ref{theorem 4.20.11.08}.

To formulate our main  result about
acceleration for parabolic equations
 we fix an integer $k\geq0$ and 
set  
\begin{equation}                       \label{11.25.11.08}
\bar u_h=\sum_{j=0}^{k}b_ju_{2^{-j}h}, , 
\end{equation}
where, naturally, $u_{2^{-j}h}$   are 
the solutions to \eqref{equation}, 
with $2^{-j}h$ in place of $h$, 
\begin{equation}                       \label{1.26.11.08}
(b_0,b_1,...,b_k)
:=(1,0,0,...,0)V^{-1}
\end{equation}
and $V^{-1}$ is the inverse of the 
Vandermonde matrix with entries
$$
V^{ij}:=2^{-(i-1)(j-1)}, \quad i,j=1,...,k+1.
$$

 The following result is 
a simple corollary
of Theorem \ref{theorem 6.25.08.08}.

\begin{theorem}                  \label{theorem 1.08.02.08} 
In each situation when Theorem \ref{theorem 6.25.08.08}
is applicable we have that the estimate
\begin{equation}                    \label{2.08.02.08}                
|\bar u_h(t,x)-u^{(0)}(t,x)|\leq N 
 (\|f\|_m+\|g\|_m) h^{k+1}
\end{equation}                  
holds for all $(t,x)\in H_T$, 
$h\in(0,h_0]$, where  
$N$ is a constant depending only on $d$, $m$,   
$ | \Lambda_1 | $, 
  $M_{ 0},\dots,M_m$, and $T$.
\end{theorem}
\begin{proof}
By  Theorem \ref{theorem 6.25.08.08}
$$
u_{2^{-j}h}=
u^{(0)}+\sum_{i=1}^k\frac{h^i}{i!2^{ji}}
u^{(i)}+ \bar r _{2^{-j}h}
 h^{k+1} ,
\quad j=0,1,...,k, 
$$
with $\bar r_{2^{-j}h}:=
2^{-j(k+1)}r_{2^{-j}h}$\,,
which gives 
$$
\bar u_h=\sum_{j=0}^kb_ju_{2^{-j}h}
=(\sum_{j=0}^{k}b_j)u^{(0)}
+\sum_{j=0}^{k}\sum_{i=1}^kb_j
\frac{h^i}{i!2^{ij}}u^{(i)}
+\sum_{j=0}^kb_j \bar r _{2^{-j}h}
 h^{k+1} 
$$
$$
=u^{(0)}+\sum_{i=1}^k\frac{h^i}{i!}u^{(i)}
\sum_{j=0}^{k}\frac{b_j}{2^{ij}}
+\sum_{j=0}^kb_j \bar  r_{2^{-j}h}
=u^{(0)}+\sum_{j=0}^kb_j \bar r _{2^{-j}h}
 h^{k+1} ,
$$
since 
$$
\sum_{j=0}^{k}b_j=1,\quad  
 \sum_{j=0}^{k} b_j 2^{-ij} =0,\quad  
i=1,2,...k
$$
 by the definition of $(b_0,...,b_k)$. 
Hence, 
$$
\sup_{H_T}|\bar u_h-u^{(0)}|
=\sup_{H_T}|
\sum_{j=0}^kb_j \bar r _{2^{-j}h}| h^{k+1} 
\leq N 
 (\|f\|_m+\|g\|_m) h^{k+1}, 
$$
and the theorem is proved.
 \end{proof}

Sometimes  it suffices  to   
combine fewer terms  $u_{2^{-j}h}$   
to get accuracy of order $k+1$. 
To consider such a case  
for odd integers $k\geq1$ define 
\begin{equation}                      \label{12.25.11.08}
\tilde{u}_{h}=\sum_{j=0}^{\tilde k}
\tilde b_ju_{2^{-j}h}\,,
\end{equation}
where  
\begin{equation}                       \label{9.25.11.08}
( \tilde b_0, \tilde b_1,...,
\tilde b_{\tilde k})
:=(1,0,0,...,0)\tilde V^{-1}, 
\quad \tilde k=\tfrac{k-1}{2}, 
\end{equation}
and  $\tilde V^{-1}$ is the inverse of the 
Vandermonde matrix with entries
$$
\tilde V^{ij}:=4^{-(i-1)(j-1)}, \quad i,j=1,...,\tilde k+1.
$$
 
\begin{theorem}
                                   \label{theorem 08.11.2.3}  

Suppose that the assumptions
of Theorem \ref{theorem 6.25.08.08} are satisfied
and   condition (iii) is  met. 
 Then for $\tilde{u}_{h}$ we have  
$$
\sup_{H_T}|u^{(0)}- \tilde u _h|\leq N 
 (\|f\|_m+\|g\|_m) h^{k+1} 
$$
for all $h\in(0,h_0]$, where $N$ depends only on 
$d$, $m$,   
$ |\Lambda_1|$, 
  $M_{ 0},\dots,M_m$, and~$T$.
\end{theorem} 

 \begin{proof}
We obtain this result from 
Theorem \ref{theorem 6.25.08.08} by 
a straightforward modification of the proof 
of the previous result, taking into account 
that for odd $j$ the terms with 
$h^j$ vanish in expansion 
\eqref{08.11.2.9} when condition (iii) 
holds in Theorem \ref{theorem 6.25.08.08}. 
\end{proof} 
 
\begin{example}
Assume that in the situation
of Theorem \ref{theorem 08.11.2.3} we have
   $m=8$. Then
$$
 \tilde u _h:=\tfrac{4}{3}u_{h/2}-\tfrac{1}{3}u_h
$$ 
satisfies
$$
\sup_{H_T}|u^{(0)}- \tilde u _h|
\leq N h^{ 4}
$$
for all $h\in(0,h_0]$. 
\end{example}

The above results show that if the data 
in equation \eqref{08.11.2.3} 
are sufficiently smooth, then the order of accuracy 
in approximating the solution $u^{(0)}$ 
can be as high as we wish if we use suitable mixtures 
of finite difference approximations calculated along 
nested grids with different mesh-sizes. 
Assume now that we need to approximate not only 
$u^{(0)}$ but its derivative $D^{\alpha}u^{(0)}$ 
for some multi-index $\alpha$ as well. What 
accuracy can we achieve? 
 The answer is closely related to the question 
whether  the expansion 
\begin{equation}
                                               \label{6.25.11.08}                     
D^{\alpha}u_h(t,x)=D^{\alpha}u^{(0)}(t,x)
+\sum_{1\leq j\leq k}\frac{h^j}{j!}D^{\alpha}u^{(j)}(t,x)
+h^{k+1}D^{\alpha}r_h(t,x) 
\end{equation}
 holds for all $(t,x)\in H_T$ and $h\in(0,h_{0}]$,  such that  
\begin{equation}                       
                                                \label{7.25.11.08}
|D^{\alpha}r_h(t,x)|\leq N (\|f\|_m+\|g\|_m)
\end{equation}
for all $(t,x)\in H_T$, $h\in(0,h_{0}]$.

The result concerning this expansion and 
the  following series of results appeared
after the authors tried to extend the above 
theorems 
from the parabolic  to the elliptic case.
The main and rather hard obstacle is that
the constants in our estimates depend on $T$
and, actually, may grow exponentially in $T$.
By the way, this obstacle
is caused by possible degeneration
of our equations and
exists even if we consider equations
in bounded smooth domain. 

To be able to give some conditions
under which this does not happen, we introduce
new notation and investigate smoothness
properties of $u_{h}$ with respect to $x$.
As a simple byproduct of this investigation
we also obtain smoothness of $u_{h}$ with respect
to $h$, which, by the way, cannot be derived from
\eqref{08.11.2.9}.

Take a  function $\tau_{\lambda}$ defined on 
$\Lambda_1$ taking values in $[0,\infty)$  and 
for  $\lambda\in\Lambda_{1}$ 
introduce  
the operators
$$
T_{h,\lambda}\varphi(x)=\varphi(x+h\lambda),
\quad
\bar\delta_{h,\lambda}
=\tau_{\lambda}h^{-1}(T_{h,\lambda}-1).  
$$ 
 Set 
$$
\|\Lambda_1\|^2
=\sum_{\lambda\in\Lambda_1}|\tau_{\lambda}\lambda|^2. 
$$  
For uniformity of notation we also introduce
$\Lambda_{2}$ as the set of fixed
 distinct vectors $\ell^{1},...,
\ell^{d}$ none of which is in $ 
\Lambda_{1}$ and   define
$$
\bar{\delta}
_{h,\ell^{i}}=\tau_{0}D_{i},
\quad 
T_{h,\ell^{i}}=1,\quad
\Lambda=\Lambda_{1}\cup\Lambda_{2},
$$
where $\tau_0>0$ is a fixed parameter.  
For $\lambda=(\lambda^{1}, \lambda^{2})\in\Lambda^{2}
 $
introduce the operators
$$
T_{h,\lambda}=T_{h,\lambda^{1}} 
T_{h,\lambda^{2}},\quad
\bar\delta_{h,\lambda} 
= \bar\delta _{h,\lambda^{1}} 
 \bar\delta_{h,\lambda^{2}} .
$$

For $k=1,2$, $\mu\in\Lambda^{k}$  we set
$$
Q_{h,\mu}\varphi=h^{-1}\sum_{\lambda\in\Lambda_{1}}
( \bar\delta _{h,\mu}q_{\lambda})\delta_{\lambda}
\varphi,\quad
L^{0}_{h,\mu}\varphi=Q_{h,\mu}\varphi +\sum_{\lambda\in\Lambda_{1}}
( \bar\delta _{h,\mu}p_{\lambda})\delta_{\lambda}
\varphi  ,
$$
$$
A_{h}(\varphi)=2\sum_{\lambda\in\Lambda }
( \bar\delta _{h,\lambda}
\varphi)L^{0}_{h,\lambda}T_{h,\lambda}\varphi,\quad
\cQ_{h}(\varphi)=\sum_{\lambda\in\Lambda_{1} }
\chi_{h,\lambda}
(\delta_{h,\lambda}\varphi)^{2}.
$$

Below $B(\bR^{d})
$ is the set of  bounded Borel 
functions on
$\bR^{d}$ and $\frK$ is 
the set of bounded operators
$\cK_{h}=\cK_{h}(t)$ 
mapping $B(\bR^{d})
$ into itself preserving  the cone
of nonnegative functions 
and satisfying $\cK_{h}1\leq1$.

Finally, fix   some constants 
$\delta\in(0,1]$ and   $K\in[1,\infty)$. 

\begin{assumption}
                              \label{assumption 08.11.2.1} 
There exists a constant $c_{0}>0$ such that $c\geq c_{0}$.
\end{assumption}

\begin{remark}
                                     \label{remark 07.9.18.8}
The above assumption  is almost irrelevant
if we only consider \eqref{equation} on a finite time interval.
Indeed,  if $c$ is just bounded, say $|c|\leq C=\text{const}$, by 
introducing a new function $v(t,x)=u(t,x)e^{-2Ct}$
we will have an equation for  $v$ similar to 
\eqref{equation} with $L_{h}^{0}v-(c+2C)v$ and $fe^{-2Ct}$
in place of $L_{h}u$ and $f$, respectively. 
Now for the new $c$ we have
  $c+2C\geq C$.
\end{remark}

\begin{assumption}
                              \label{assumption 11.22.11.06} 
 
We have $m\geq1$ and 
for any $h\in(0,h_{0}]$, there
exists  an operator $\cK_{h}=\cK_{h,m}\in\frK$,
 such that  
\begin{equation}
                                             \label{3.24.1}
m A_{h}(\varphi)
\leq(1- \delta)\sum_{\lambda\in\Lambda}
\cQ_{h}( \bar\delta _{h,\lambda}\varphi)
+K\cQ_{h}(\varphi)
+2(1-\delta)c\cK_{h}\big(\sum_{\lambda\in\Lambda }
| \bar\delta _{h,\lambda}\varphi|^{2}\big) 
\end{equation}
on $H_{T}$
for all smooth functions $\varphi $.
\end{assumption}

\begin{assumption}         \label{assumption 07.10.16.1}  
We have $m\geq2$ and, for any $h\in(0,h_{0}]$
and $n=1,...,m$, there
exists  an operator $\cK_{h}=\cK_{h,n}\in\frK$,
 such that  
$$
n \sum_{\nu\in\Lambda}A_{h}( \bar\delta _{h,\nu}\varphi)
+ n(n-1)\sum_{\lambda\in\Lambda^{2}}
( \bar\delta _{h,\lambda}\varphi)
Q_{h,\lambda}T_{h,\lambda}\varphi
\leq (1- \delta)\sum_{\lambda\in\Lambda^{2}}
\cQ_{h}( \bar\delta _{h,\lambda}
\varphi)
$$
\begin{equation}
                                           \label{3.24.01}
+K
\sum_{\lambda\in\Lambda}
\cQ_{h}( \bar\delta _{h\lambda}\varphi)
+2(1-\delta)c\cK_{h}\big(\sum_{\lambda\in\Lambda ^{2}}
| \bar\delta _{h,\lambda}\varphi|^{2}\big) 
+K\cK\big(\sum_{\lambda\in\Lambda}
| \bar\delta _{h,\lambda}\varphi|^{2}
\big)
\end{equation}
on $H_{T}$
for all smooth functions $\varphi $.
\end{assumption}
Obviously Assumptions 
\ref{assumption 11.22.11.06} and \ref{assumption 07.10.16.1} 
are satisfied if $q_{\lambda}$ and $p_{\lambda}$
are independent of $x$. In the general case,
as it is discussed in 
\cite{GK1}, the above assumptions impose not only 
analytical conditions, but they are related also 
to some 
structural conditions, 
which can somewhat easier be analized under  the
symmetry condition (S).

\begin{assumption}            \label{assumption 18.12.07.06} 
For   all $t\in[0,T]$ 
\begin{equation}
                                         \label{3.9.1}
\sum_{\lambda\in\Lambda_{1}}\lambda q_{\lambda}(t,x )
\quad\hbox{is independent of}\quad   x .
\end{equation}
 \end{assumption}
 
In the main case of applications we will require
the last sum to be identically zero 
as in Assumption \ref{assumption 14.23.01.08}.

\begin{remark}                         \label{remark 07.10.17.4} 
Assumptions  \ref{assumption 11.22.11.06}  and \ref{assumption
07.10.16.1}  are discussed at length
and in many
details in \cite{GK1} and \cite{GK2}, and   
sufficient conditions, without involving test functions 
$\varphi$ are given for these assumptions  to 
be satisfied. 
In particular, it is shown in \cite{GK2} 
that if condition (S) holds,  $m\geq2$, 
$\tau_{\lambda}=1$,  
Assumptions 
\ref{assumption 16.12.07.06}
and \ref{assumption 1.26.11.06} are satisfied,  
and $q_{\lambda}\geq\kappa$ for a constant 
$\kappa>0$,  then
both Assumptions \ref{assumption 11.22.11.06}
and \ref{assumption 07.10.16.1} are satisfied 
for 
any $c_{0}>0$ 
 and $\delta\in(0,1)$,  
if
$h_{0}$ is sufficiently small and 
$ \tau_0$, $K$, and $\cK$
are chosen appropriately. Moreover, 
the condition $\kappa>0$
can be dropped, provided,
additionally, that $c_{0}$
is large enough 
(this time we need not assume that $h$ is small).
Remember, that by Remark \ref{remark 07.9.18.8} 
the condition that $c_{0}$ be large is, actually,
harmless as long as we are concerned with equations
on a finite time interval.
Mixed situations, when $c$ is large
at those points where some of $q_{\lambda}$
can vanish are also considered in \cite{GK2}. 

In \cite{GK1} we have seen that  
Assumption
\ref{assumption 11.22.11.06} imposes certain nontrivial 
{\em structural\/} conditions on $q_{\lambda}$ which cannot
be guaranteed by  
 the size of $c_{0}$ if  $q_{\lambda}$ is only once 
continuously differentiable.
In contrast,  
even without condition (S), given that Assumptions  
\ref{assumption 16.12.07.06}, \ref{assumption 11.22.11.06},
\ref{assumption 18.12.07.06}  
are satisfied and $m\geq2$, as  is shown in \cite{GK2}, 
Assumption \ref{assumption 07.10.16.1}
is also satisfied if $c_{0}$ is large enough.

\end{remark}

\begin{theorem}                         \label{theorem 1.25.11.08}
Let Assumption  \ref{assumption 16.12.07.06} through 
\ref{assumption 07.10.16.1} hold with $m\geq3$.  
Let $k\geq0$ and $l\in[0,m]$ be integers. 
Then 
for every multi-index 
$\alpha$ such that $|\alpha|\leq l$ the function
$D^{\alpha}u_h$ is a continuous function on $H_T$ and  
expansion \eqref{6.25.11.08}  
holds with
$D^{\alpha}r_h$ satisfying 
 \eqref{7.25.11.08}, 
provided one of the following 
conditions  is  met:

(i) $m\geq3k+3+l$;
    
(ii) $m\geq2k+3+l$ and condition (S) holds;

(iii) $k$ is odd, $m\geq2k+2+l$, and conditions (S) 
and \eqref{08.11.7.5} are satisfied. 
In each of the cases (i)-(iii) the constant  
  $N$  depends only   on $d$, $m$, 
$\delta$, $K$, $\tau_0$, $c_0$,  
$|\Lambda_1|$, 
 $\|\Lambda_1\|$,  $M_{0},\dots,M_m$. 
In case (iii) we have $u^{(j)}=0$ 
for all odd $j$ in the expansion.  
\end{theorem}

We prove this theorem in Section \ref{section 5.23.11.08}.  
Remember the definition of $\bar u_h$ 
and $\tilde u_h$ in 
\eqref{11.25.11.08} and 
\eqref{12.25.11.08}. The following   is 
an obvious consequence of Theorem~\ref{theorem 1.25.11.08}.

\begin{corollary}                         
                                       \label{corollary 10.25.11.08}
Suppose that the assumptions of 
Theorem \ref{theorem 1.25.11.08} are satisfied. 
Then 
\begin{equation*}
\sup_{H_T}|D^{\alpha}\bar u_h-D^{\alpha}u^{(0)}|\leq Nh^{k+1}
(\|f\|_m+\|g\|_m), 
\end{equation*}
and if condition (iii) is met then
\begin{equation*} 
\sup_{H_T}|D^{\alpha}\tilde u_h-D^{\alpha}u^{(0)}|\leq Nh^{k+1}
(\|f\|_m+\|g\|_m), 
\end{equation*}
where $N$ depends only on 
on $d$, $m$, 
$\delta$, $K$, $\tau_0$, $c_0$,  
$|\Lambda_1|$, 
 $\|\Lambda_1\|$, 
  $M_{0},\dots,M_m$. 
\end{corollary}

\begin{remark}
                                            \label{remark 08.12.14.1}
Observe that for $k=0$ Theorem \ref{theorem 1.25.11.08}
implies that 
\begin{equation}
                                              \label{08.12.14.3}
\sup_{H_{T}}|D^{\alpha}u_{h}- D^{\alpha}u^{(0)}|\leq Nh
\end{equation}
if $m\geq 3+|\alpha|$ and 
 Assumption  \ref{assumption 16.12.07.06} through 
\ref{assumption 07.10.16.1} hold.
In addition one can replace $D^{\alpha}u_{h}$ in \eqref{08.12.14.3}
with $\delta^{\alpha}_{h}$, where
$$
\delta^{\alpha}_{h}=\delta_{h,e_{1}}^{\alpha_{1}}\cdot...
\cdot\delta_{h,e_{d}}^{\alpha_{d}}
$$
and $e_{i}$ is the $i$th basis vector in $\bR^{d}$.
This follows easily from the mean value theorem and Theorem
\ref{theorem 17.25.06.07}      below.
The reader understands that similar assertion is true
in case of Corollary \ref{corollary 10.25.11.08}
with the only difference that one needs larger $m$ and better
finite-difference approximations of $D^{\alpha}$.
\end{remark}

Next we investigate the smoothness of $u_h$ in 
$x$ and $h$. Recall that 
for functions $\varphi$ depending on
$h$  we use the notation 
$D_h^r\varphi$ for the $r$-th derivative of   
$\varphi$ in $h$. As usual, $D^0_h\varphi:=\varphi$.

\begin{remark}
                                    \label{remark 08.2.19.1}
Suppose that Assumption 
\ref{assumption 16.12.07.06} is satisfied.
Take an $h_{1}\in(0,h_{0})$, 
consider   equation \eqref{equation} as an equation
about a function $u_{h}(t,x)$ as function of
$(h,t,x)\in[h_{1},h_{0}]\times H_{T}$ and look for solutions
in the  space $\frB^{m}(h_{1})=\frB^{m}_{T}(h_{1})$ which is 
defined  as the space of functions on
$[h_{1},h_{0}]\times H_{T}$ with finite norm
\begin{equation}
                                                        \label{09.4.16.1}
\sum_{|\alpha|+3r\leq m}
\sup_{[h_{1},h_{0}]\times H_{T}}
|D^{\alpha}D^{r}_{ h}u_{h}(t,x)|.
\end{equation}

It is obvious that
the integrand in \eqref{equation} can be considered
as the result of application of an operator,
which is bounded in $\frB^{m}(h_{1})$, to $u_{h}(s,x)$.
Therefore, a standard abstract theorem on solvability
of ODEs in Banach spaces shows that there exists a solution
of \eqref{equation} in $\frB^{m}(h_{1})$. Since just bounded
solutions are uniquely defined by \eqref{equation},
we conclude that our $u_{h}$ belongs to $\frB^{m}(h_{1})$
for any $h_{1}\in(0,h_{0})$. Obviously, if the derivatives
of the data are continuous in $x$, the same will hold
for $u_{h}$.

The above argument, actually, works if we replace
$|\alpha|+3r\leq m$ with $|\alpha|+r\leq m$ in \eqref{09.4.16.1}.
We talk about \eqref{09.4.16.1} in the above form because we will show that
under our future assumptions
the quantity \eqref{09.4.16.1} is bounded independently of $h_{1}$. 
\end{remark}

\begin{theorem}                      \label{theorem 17.25.06.07}                
 Let $k\geq0$ and $m\geq2$ be integers and suppose that
Assumptions 
\ref{assumption 16.12.07.06}
through \ref{assumption 07.10.16.1}  are satisfied.
Then, 
for each  integer $r\geq0$ such that
$$
3k+r\leq m,
$$
the generalized derivatives $D^{r} D_{h}^{k}u _h$ exist
on $ (0,h_{0}]\times H_{T}$,   
are bounded and we have 
\begin{equation}                           \label{20.19.02.07}
 |D^{r}D_{h}^{k} u _h |
\leq N(\|f \|_{m}+\|g \|_{m}) ,
\end{equation}
  where 
$N$ 
is a constant depending only on  
$m, \delta, c_0,\tau_0,K $, 
$M_0$, ..., $M_{m}$, 
 $|\Lambda_1|$, and 
$\|\Lambda_1\|$. 
In particular, $u_{h} \in\frB^{m}$ and
$$
\|u_{h} \|_{m}\leq N(\|f \|_{m}+\|g \|_{m}).
$$
\end{theorem}

We prove this theorem in  Section \ref{section 3.7.1},   
and in Section \ref{section 1.25.11.08} we show that
the following fact, used
when we come to the elliptic case, is a simple corollary of it.
\begin{theorem}
                                      \label{theorem 08.11.3.1} 

Suppose that 
Assumptions 
\ref{assumption 16.12.07.06}
through \ref{assumption 07.10.16.1}
hold with $m\geq 2$.   
Then the constant $N$ in \eqref{08.11.3.2}
depends only on  
$m, \delta, c_0,\tau_0,K $, 
$M_{ 0 }$, ..., 
$M_{m}$, $|\Lambda_{1}|$, and
$\|\Lambda_1\| $ (thus, is independent of $T$).
The same is true for the constants $N$ in
Theorems \ref{theorem 08.11.7.2},
 \ref{theorem 6.25.08.08},
\ref{theorem 1.08.02.08}, and
\ref{theorem 08.11.2.3}.
\end{theorem}

Additional information on the behavior of
$D^{r}D_{h}^{k} u _h$ for small $h$ 
is provided by the following
result which we prove in Section \ref{section 3.7.1}.
\begin{theorem}
                                  \label{theorem 08.2.19.1}
Let $k\geq1$ be an {\em odd\/} number and suppose that 
Assumptions 
\ref{assumption 16.12.07.06} 
through 
\ref{assumption 07.10.16.1} 
hold with $m\geq 3k+1$.   Assume that the symmetry condition
(S) and \eqref{08.11.7.5} are satisfied.

Then, 
for any integer $r\geq0$ such that
$$
 3k+r\leq m-1
$$ 
we have
\begin{equation}
                                           \label{08.2.19.7}
\sup_{H_{T}}|D^{r}D_{h}^{k}u _{h}|
\leq N (\|f\|_m+\|g\|_m) h
\end{equation}
for all $h\in(0,h_{0}]$, where $N$ depends only on 
$m$, $\delta$, $c_0$, $\tau_0$, $K$, 
$|\Lambda_1|$, 
 $\|\Lambda_1\|$,
$M_0$,..., $M_{m}$.
\end{theorem}

\mysection{Main results for elliptic equations}

Here we assume that  $p_{\lambda}$, $q_{\lambda}$, 
$c$, and $f$ are independent of $t$
and turn now our attention to  the 
equations 
\begin{equation}                             \label{11.27.01.08}
L_hv_{h}(x)+f (x)=0\quad x\in\bR^d,
\end{equation}
\begin{equation}                             \label{08.11.3.4}
\cL v(x)+f (x)=0\quad x\in\bR^d.
\end{equation}

Naturally by a solution of \eqref{08.11.3.4} we 
mean
a function $v$ on $\bR^d$ 
such that it belongs to $\frB^2$ and 
\eqref{08.11.3.4} holds almost everywhere. 
Clearly, if a solution $v$ belongs to 
$\frB^3$ and $q_{\lambda}$, 
$p_{\lambda}$, $c$, and $f$ are continuous functions on 
$\bR^d$, then \eqref{08.11.3.4} holds 
everywhere.

First we prove the existence 
and uniqueness of the solutions of  
equations \eqref{11.27.01.08} and \eqref{08.11.3.4}.

\begin{theorem}                              \label{theorem 1.20.11.08}
Suppose that 
Assumption  \ref{assumption 16.12.07.06} 
is satisfied 
with an $m\geq0$ 
 and 
let Assumptions \ref{assumption 1.26.11.06} 
and \ref{assumption 08.11.2.1} 
hold.  
Then 
equation \eqref{11.27.01.08} has a unique 
bounded solution $v_h$. Moreover, $v_h$ belongs to 
$\frB^m$.  
\end{theorem} 
 
\begin{proof}
Observe that \eqref{11.27.01.08}
is equivalent to
$$
v_{h}(x)=h^{2}\xi(x)
f(x)+\xi(x)\sum_{\lambda\in\Lambda_{1}}
\chi_{\lambda}v_{h}(x+\lambda h),
$$
where 
$$
\xi^{-1} =h^{2}c +\sum_{\lambda\in\Lambda_{1}} \chi_{\lambda} . 
$$
It is seen that the existence and uniqueness
of bounded solution of \eqref{11.27.01.08} follows
by contraction principle. Using smooth successive
iterations yields that $v_{h}\in\frB^{m}  $.
\end{proof}

\begin{theorem}                        \label{theorem 1.22.11.08} 
 Let  Assumptions 
\ref{assumption 16.12.07.06}
through \ref{assumption 07.10.16.1}
hold with an $m\geq 2$.  Then equation 
\eqref{08.11.3.4} has a unique  solution $v$
in the space $\frB^{2}$.  
Moreover, $v\in\frB^{m}$ and there is a constant $N$  
depending only on  
$m$, $ \delta$, $c_0$, $\tau_0$, $K $, 
$M_0$, ..., $M_{m}$, $|\Lambda_{1}|$, and
$\|\Lambda_1\| $ such that 
\begin{equation}                           \label{1.23.11.08}
\|v\|_m\leq N\|f\|_m. 
\end{equation}
\end{theorem}

\begin{proof}        
First we prove uniqueness. Let $v\in\frB^{2}$ satisfy
\eqref{08.11.3.4} with $f=0$. Take a constant $\nu>0$, 
so small  
that $c-\nu\geq c_0/2$ and conditions 
\eqref{3.24.1} and \eqref{3.24.01}
 hold  with $c-\nu$ and $\delta/2$ in place
of $c$ and $\delta$, respectively. 
Then for each $T>0$ 
the function $u(t,x):=e^{\nu t} v(x) )$, 
  $(t,x)\in H_{T}$,  
is a  solution of class $\frB^{2}_{T}$ of
the equation
\begin{equation}                       \label{8.22.11.08}
\frac{\partial}{\partial t}u=(\cL +\nu)u 
\quad\text{on $H_T$}
\end{equation}
with initial condition 
$u(0,x)=v(x) $. 
Hence by virtue of Theorem \ref{theorem 08.11.3.1}
for every 
$T>0$ 
$$
e^{\nu T}|v(x) |=|u(T,x)|\leq N \|v \|_2,  
$$ 
where $N$ is independent of $(T,x)$.
 Multiplying both sides 
of the above inequality by $e^{-\nu T}$ and letting
$T\to\infty$ we get $v=0$, which proves   uniqueness.

To show the existence of a solution in $\frB^{m}$, 
let $u$ be a function
defined on 
$H_{\infty} $ 
such that for each $T>0$ its restriction 
onto $H_T$ is the unique 
solution in $\frB^{m}_{T}$ of \eqref{8.22.11.08} with
initial  condition $u(0,x)=f(x)$ 
(see Theorem \ref{theorem 08.11.2.1}).   
By Theorem \ref{theorem 08.11.3.1} 
$$
\sup_{H_{\infty}}
 \sum_{r\leq m}|D^{r}u|\leq N\|f\|_m 
$$
with a constant $N$ depending only on  
$m$, $ \delta$, $c_0$, $\tau_0$, $K $,
$M_0$, ..., $M_{m}$, $|\Lambda_{1}|$, and
$\|\Lambda_1\| $. 
Hence  
$$
v(x):=\int_{0}^{\infty}e^{-\nu t}u(t,x)\,dt, \quad x\in\bR^d 
$$
is a well-defined function on $\bR^d$, 
$v\in\frB^m$, and 
$$
\cL v(x)=\int_{0}^{\infty}e^{-\nu t}\cL u(t,x)\,dt
$$
$$
=
\int_{0}^{\infty}e^{-\nu t}
(\frac{\partial}{\partial t}u(t,x)-\nu u(t,x) )\,dt
=-f(x), 
$$
where the last equality is obtained by 
integration by parts. 
Consequently, $v$ is a solution of \eqref{8.22.11.08} 
and it satisfies  estimate 
\eqref{1.23.11.08}.
\end{proof}
 
\begin{theorem}                      \label{theorem 17.25.06.08}                
 Let $k\geq0$ and  suppose that   
Assumptions \ref{assumption 16.12.07.06} 
through \ref{assumption 07.10.16.1} are satisfied
with an $m\geq 3k$.  
Then, 
  for any $h\in(0,h_{0}]$
and 
for each integer $r\geq0$, such that
$$
3k+r\leq m,
$$
  for the unique bounded 
solution $v_h$ of \eqref{11.27.01.08} 
 we have
\begin{equation}                           \label{20.19.02.08}
\sup_{(0,h_{0}]\times\bR^{d}}|D^{r}D_{h}^{k} v _h |
\leq N\|f \|_{m},
\end{equation}
  where 
$N$ 
is a constant depending only on  
$m, \delta, c_0,\tau_0,K $, 
$|\Lambda_1|$, $\|\Lambda_1\|$, 
$M_{ 0}$, ..., $M_{m}$. 
In particular, 
$$
\|v_{h} \|_{m}\leq N\|f \|_{m}.
$$
\end{theorem}

\begin{proof} 
 To prove \eqref{20.19.02.08},
take a constant 
$\nu>0$ as in the proof of
Theorem \ref{theorem 1.22.11.08},
define 
$u(t,x):=v_h(x)e^{\nu t} $, and observe that
  $u$ is the unique bounded solution of 
$$
\frac{\partial}{\partial t}u
=L^{0}_{ h }u-(c-\nu)u+e^{\nu t}f, 
 \quad u(0,x)=v_h(x) .
$$
By Theorem \ref{theorem 17.25.06.07} 
  for any $T>0$
$$
e^{\nu T}|D^rD_{h}^{k} v_h  (x)|
=|D^rD_{h}^{k}u (T,x)|
\leq
N e^{\nu T}\|f\|_{m}
+
N\|v_h\|_{m}, 
$$
where $N$ is a constant, depending only on 
$m, \delta, c_0,\tau_0,K $, 
$|\Lambda_1|$, $\|\Lambda_1\|$, 
$M_0$, ..., $M_{m}$. 
By multiplying the extreme terms by $e^{-\nu T}$
and letting $T\to\infty$, we get the result. 
\end{proof}

From  estimate \eqref{08.2.19.7} we obtain 
the corresponding estimate for the derivatives 
of $v_h$. 

\begin{theorem}
                                  \label{theorem 09.24.08.08}
Let the conditions of Theorem 
\ref{theorem 08.2.19.1} hold.   
Then 
for any   
integer $r\geq0$ such that
$$
3k+r\leq m-1,
$$
for the solution $v_h$ of \eqref{11.27.01.08}
we have 
\begin{equation*}
                                           \label{13.24.08.08}
\sup_{\bR^{d}}|D^{r}D_{h}^{k}v _{h}|
\leq N \|f\|_m h
\end{equation*}
for all $h\in(0,h_{0}]$, where $N$ depends only on 
$m, \delta, c_0,\tau_0,K $, 
$|\Lambda_1|$, 
 $\|\Lambda_1\|$  and $M_0$,..., $M_{m}$.
\end{theorem}
\begin{proof}
This theorem can be deduced from Theorem  
\ref{theorem 08.2.19.1} in the same way as 
Theorem \ref{theorem 17.25.06.08} is obtained  
from Theorem \ref{theorem 17.25.06.07}. 
\end{proof}

Now we want to establish an expansion for 
$v_h$, i.e., to show for an integer $k\geq0$ 
the existence of some functions
 $v^{(0)}$,...,$v^{(k)}$  
on $\bR^d$, and a function 
$R_h$ on $\bR^d$ for each $h\in(0,h_0]$ 
such that 
for all $x\in\bR^d$ 
and $h\in(0,h_0]$ 
\begin{equation}                           \label{2.20.11.08}
v_h(x)=v^{(0)}(x)+\sum_{1\leq j\leq k}
\frac{h^j}{j!}v^{(j)}(x)+h^{k+1}R_h(x), 
\end{equation}
\begin{equation}                         \label{3.20.11.08}
\sup_{h\in(0,h_0]}\sup_{\bR^d}
|R_h|\leq N\|f\|_m\quad 
\end{equation}         
with a constant $N$. 
 
\begin{theorem}                    \label{theorem 4.20.11.08} 
Suppose that
Assumptions  \ref{assumption 16.12.07.06} through 
\ref{assumption 07.10.16.1} 
are satisfied with an $m\geq3$.   
Let $k\geq0$ be an integer. 
Then expansion  \eqref{2.20.11.08}   
holds  with $v^{(0)}$ being the unique  $\frB^{m}$  
solution of \eqref{08.11.3.4} and 
$R_{h}$ satisfying  
 \eqref{3.20.11.08} 
provided one of the following conditions is met:

(i) $m\geq3k+3$;
    
(ii) $m\geq2k+3$ and condition (S) holds;

(iii) $k$ is odd, $m\geq2k+2$, and conditions (S) 
and \eqref{08.11.7.5} are satisfied. 

\noindent
In each of the cases (i)-(iii) the constant  
  $N$ in \eqref{3.20.11.08}  
depends only   on $d$, $m$,   
$\delta, c_0,\tau_0,K $, 
$|\Lambda_1|$, $\|\Lambda_1\|$,
  $M_{ 0},\dots,M_m$. 
Moreover, when (iii) holds we have
$v^{(j)}=0$ for all odd $j$.
\end{theorem}

\begin{proof}
Take a small constant $\nu>0$, 
as in the proof of
Theorem \ref{theorem 1.22.11.08},  
let $u$ be   a function 
defined on 
$H_{\infty} $ 
such that for each $T>0$ its restriction 
onto $H_T$ is the unique 
solution in $\frB^{m}_{T}$ of  
$$
\frac{\partial}{\partial t}u_h=(L_h+\nu)u_h 
\quad (t,x)\in H_{\infty}
$$
$$
u_h(0,x)=f(x)\quad x\in\bR^d ,
$$
(see Remark \ref{remark 9.19.01.08}). As in the proof of
Theorem \ref{theorem 1.22.11.08} we get that
$$
v_h(x) =\int_0^{\infty}e^{-\nu t}u_h(t,x)\,dt.
$$
By Theorem \ref{theorem 6.25.08.08} 
in each of the cases (i)-(iii) we have 
\begin{equation}                       \label{6.20.11.08}
u_h(t,x)=u^{(0)}(t,x)
+\sum_{1\leq j\leq k}\frac{h^j}{j!}u^{(j)}(t,x)
+h^{k+1}r_h(t,x), 
\end{equation}
for all $(t,x)\in H_{\infty}$, $h\in(0,h_0]$, 
and by Theorem  \ref{theorem 08.11.3.1}   
we have 
\begin{equation}                            \label{2.23.11.08}
\sup_{h\in(0,h_0]}\sup_{H_{\infty}}
\{|u_h|+\sum_{j=0}^{k}|u^{(j)}|+|r_h|\}\leq N\|f\|_m
\end{equation}
with a constant $N$ depending only on 
$d$, $m$,   
$\delta, c_0,\tau_0,K $, 
$M_0$,...,$M_m$, $|\Lambda_1|$ and $\|\Lambda_1\|$. 
Multiplying both sides of equation 
\eqref{6.20.11.08} by $e^{-\nu t}$ and then 
integrating them over $[0,\infty)$ with respect to $dt$,  
we get 
expansion \eqref{2.20.11.08} with 
$$
R_h(x):=\int_0^{\infty}e^{-\nu t}r_h(t,x)\,dt, 
$$
$$
v^{(j)}(x):=\int_0^{\infty}e^{-\nu t}u^{(j)}(t,x)\,dt, 
\quad \text{for $j=0,\dots,k$}.
$$
Clearly, \eqref{2.23.11.08} implies that 
\eqref{3.20.11.08} holds with $N$ depending only on 
$d$, $m$,   
$\delta, c_0,\tau_0,K $, 
$M_{ 0}$,...,$M_m$, $|\Lambda_1|$, and $\|\Lambda_1\|$. 
As we know the function $u^{(0)}$
in \eqref{6.20.11.08} is the
$\frB^{m}$ solution of 
$$
\frac{\partial}{\partial t}u=(\cL+\nu)u 
\quad (t,x)\in H_{\infty}, 
$$
$$
u(0,x)=f(x)\quad x\in\bR^d,  
$$ 
which as we have seen in the proof of 
Theorem \ref{theorem 1.22.11.08}
guarantees  
that $v^{(0)}$ is the unique $\frB^{m}$ solution 
of equation \eqref{08.11.3.4}.
\end{proof}
\begin{remark}
 We can show 
similarly that $v^{(i)}$, $i=1,...,k$, is the unique 
  solution of the system
$$
\cL v^{(j)}(s,x)
+\sum_{i=1}^{j}C^{i}_{j}\cL^{(i)}v^{(j-i)} 
 =0
$$
in an appropriate class of functions (cf. 
Theorem \ref{theorem 08.11.7.2}).
\end{remark}

The following result can be obtained easily from 
Theorem \ref{theorem 1.25.11.08} by inspecting the proof 
of the previous theorem.

\begin{theorem}                    \label{theorem 8.25.11.08}
Let $p_{\lambda}$, $q_{\lambda}$, $c$, and 
$f$ satisfy the conditions of 
Theorem \ref{theorem 4.20.11.08}, with $m-l$ in place of 
$m$ in each of the conditions (i)--(iii) for an integer 
$l\in[0,m]$. Then $D^{\alpha}v_h$ 
is a bounded continuous function 
on $\bR^d$ for every multi-index 
$\alpha$, $|\alpha|\leq l$, and the  
expansion \eqref{2.20.11.08} is valid 
with $D^{\alpha}v_h$, $\{D^{\alpha}v^{(j)}\}_{j=0}^k$ 
and $D^{\alpha}R_h$ in place of 
$v_h$, $\{v^{(j)}\}_{j=0}^k$ 
and $R_h$, respectively. 
Furthermore, 
\eqref{3.20.11.08} holds with $D^{\alpha}R_h$ 
in place of $R_h$ and a constant 
$N$  depending only   on $d$, $m$,   
$\delta, c_0,\tau_0,K $, 
$|\Lambda_1|$, $\|\Lambda_1\|$,
  $M_{0},\dots,M_m$. 
In case (iii) we have $v^{(j)}=0$ 
for all odd $j$ in the expansion. 
\end{theorem}

Set 
\begin{equation*}
\bar{v}_{h}=\sum_{j=0}^{k}
b_jv_{2^{-j}h}\,, 
\quad 
\tilde{v}_{h}=\sum_{j=0}^{\tilde k}
\tilde b_jv_{2^{-j}h}\,,
\end{equation*}
where $(b_0, b_1,\dots,b_{k})$ 
and $\tilde k$, 
$( \tilde b_0, \tilde b_1,\dots,\tilde b_{\tilde k})$ 
are defined in \eqref{1.26.11.08} and in \eqref{9.25.11.08}. 
Then we have the following corollary.

\begin{corollary}                         
                                       \label{corollary 3.08.02.08}
Suppose that the assumptions of 
Theorem \ref{theorem 8.25.11.08} are satisfied. 
Then for every multi-index 
$\alpha$ with $|\alpha|\leq l$,
\begin{equation*}
\sup_{\bR^d}|D^{\alpha}\bar v_h-D^{\alpha}v^{(0)}|\leq N\|f\|_mh^{k+1}, 
\end{equation*}
and if condition (iii) is met then
\begin{equation*} 
\sup_{\bR^d}|D^{\alpha}\tilde v_h-D^{\alpha}v^{(0)}|\leq N\|f\|_mh^{k+1}, 
\end{equation*}
where $N$ depends only on 
on $d$, $m$, 
$\delta$, $K$, $\tau_0$, $c_0$,  
$|\Lambda_1|$, 
 $\|\Lambda_1\|$, 
  $M_{0},\dots,M_m$. 
\end{corollary}

\mysection{Proof of uniqueness in Theorem
\protect\ref{theorem 08.11.2.1} and a stipulation}
                                           \label{section 08.12.17.1}

We will see later that the proof of 
Theorem \ref{theorem 6.25.08.08} 
only uses the existence of sufficiently smooth
solutions of \eqref{08.11.2.3} and \eqref{08.11.2.6}. Therefore,
if $m\geq3$, uniqueness of $u^{(0)}$ follows from expansion
\eqref{08.11.2.9}.  
If $m=2$,
one can use simple ideas based on integrating
by parts. We briefly outline these ideas
referring for details to \cite{Ol1}, \cite{Ol}, \cite{OR}.

First, one may assume that $g=f=0$ and let $u^{(0)}$ 
be the corresponding 
solution.
Then, by introducing a new function
$v=u^{(0)}(\cosh |x|)^{-1}$ one reduces the issue
to uniqueness of $v$, which satisfies
an equation similar to \eqref{08.11.2.1}
with $g=f=0$ and different coefficients which we denote by
$\hat{q}_{\lambda}$, $\hat{p}_{\lambda}$, 
and $\hat c=c$, and, moreover, 
$v,Dv,D^{2}v
\in L_{2}(H_{T})$. After that one multiplies the
equation for $v$ by $v$ and integrates
over $H_{T}$. One uses integration by parts, 
  and  the fact that
due to the assumption $q_{\lambda}\geq0$ we have
$|D\hat q_{\lambda}|^{2}\leq 4
\hat q_{\lambda}\sup|D^{2}\hat q_{\lambda}|$.
One also uses
  Young's inequality implying that
$$
|v(\partial_{\lambda}\hat q_{\lambda})\partial_{\lambda}
 v|\leq  N|v\hat q_{\lambda}^{1/2}
\partial_{\lambda}
v|\leq  \hat q_{\lambda}  
(\partial_{\lambda}v)^{2}+Nv^{2}, 
$$
and the fact that $2\hat{v}\hat p_{\lambda}
\partial_{\lambda} \hat{v}=
\hat p_{\lambda}
\partial_{\lambda} (\hat{v})^{2}$.
Then  one quickly arrives at a relation like
$$
\int_{H_{T}} (N-c)|v|^{2}\,dxdt
\geq\int_{\bR^{d}}|v(T,x)|^{2}\,dx\geq0,
$$
where $N$ is a constant   independent of $c$.
If $c$ is large enough, the above inequality is
only possible if $v=0$, which proves
uniqueness if $c$ is large enough. In the general case
it only remains to observe that the
usual change of the unknown function taking
$v(t,x)e^{\lambda t}$ 
in place of $v$ for an appropriate $\lambda$
will lead to as large $c$ as we like.

\begin{remark}
                                    \label{remark 08.11.7.1}
Notice that apart from uniqueness in Theorems
\ref{theorem 08.11.2.1} and \ref{theorem 08.11.7.2}
 all our other assertions
and assumptions are stable under applying mollifications
of the data with respect to $x$. For instance, take
a nonnegative $\zeta\in C^{\infty}_{0}(\bR^{d})$
with unit integral, for $\varepsilon>0$ define
$\zeta_{\varepsilon}(x)=\varepsilon^{-d}\zeta(x/\varepsilon)$
and for locally summable $\psi(x)$ use the notation
$$
\psi^{(\varepsilon)}=\psi*\zeta_{\varepsilon}.
$$
Then $q_{\lambda}^{(\varepsilon)},p_{\lambda}^{(\varepsilon)},
c^{(\varepsilon)},f^{(\varepsilon)}$, and $g^{(\varepsilon)}$
will satisfy the same assumptions with the same
constants as the original ones 
and will be infinitely differentiable in $x$.

It is not hard to see that if our assertions are true
for the mollified data, then they are also true for the original
ones. For instance, 
let $v^{\varepsilon}$ be the solution
of \eqref{08.11.2.1} with the new data. The uniform in 
$\varepsilon$ estimates of the derivatives in $x$ and the equation
itself, guaranteeing that the first derivatives in time 
are bounded, show that $v^{\varepsilon}$ are uniformly
continuous in $[0,T]\times\{|x|\leq R\}$ for any $R$. Then
there is a sequence $\varepsilon_{n}\downarrow0$
such that $v^{\varepsilon_{n}}$ converges
uniformly in $[0,T]\times\{|x|\leq R\}$ for any $R$
to a bounded continuous function~$v$. 

This along with uniform
boundedness of $|D^{\alpha}v^{\varepsilon}|$, $|\alpha|\leq m$,
lead  to the fact that the generalized derivatives
$|D^{\alpha}v |$, $|\alpha|\leq m$, are bounded and admit
the same estimates as those of $v^{\varepsilon}$.
Also since $D^{\alpha}v^{\varepsilon_{n}}\to
D^{\alpha}v $ in the sense of distributions and all
of them are uniformly bounded, we conclude that this convergence
is true in the weak sense in any
$L_{2}([0,T]\times\{|x|\leq R\})$. Now it is easy  to
pass to the limit
in equation
\eqref{08.11.7.1} written for modified coefficients
and $v^{\varepsilon}$ in place of $u$ concluding that
 since the derivatives 
converge weakly and $q^{(\varepsilon)}_{\lambda}
\to q_{\lambda}$,..., $f^{(\varepsilon)}\to f$ uniformly
on $H_{T}$, $v$ satisfies \eqref{08.11.7.1}.

Similar argument takes care of Theorem \ref{theorem 08.11.7.2}
(in which uniqueness will be derived from uniqueness
in Theorem \ref{theorem 08.11.2.1}).

Our claim about stability of other results is almost 
obvious and\smallskip

\noindent{\em from this moment on we will assume that
the data are as smooth in $x$ as we like.}

\end{remark}
  
\mysection{Proof of Theorems  
\ref{theorem 17.25.06.07} 
and \ref{theorem 08.2.19.1}}                         
                                         \label{section 3.7.1}

In \cite{GK1} (see there Theorems 2.3 and 2.1 
and  Corollary 3.2 if $m=0$)
 and 
\cite{GK2} we obtained the following result on the 
smoothness in $x$ of the solution $u_h$ to 
equation \eqref{equation}.

\begin{theorem}
                                          \label{theorem 4.7.2}

Suppose that Assumptions 
\ref{assumption 16.12.07.06} and
\ref{assumption 08.11.2.1} are satisfied. 
Suppose that
(i) if $m=1$, then Assumptions 
\ref{assumption 1.26.11.06} 
and \ref{assumption 11.22.11.06} are satisfied,
and (ii) if $m\geq2$, then
Assumptions 
\ref{assumption 1.26.11.06},
\ref{assumption 11.22.11.06},  
 \ref{assumption 07.10.16.1}, and \ref{assumption 18.12.07.06}  
are satisfied.
Then for $h\in(0,h_{0}]$
we have that $D^{k}u_{h}$, $k=0,...,m$, are continuous in $x$
and
\begin{equation}
                                             \label{4.8.03}
 \sup_{H_{T}}\sum_{k=0}^{m} 
|D^{k}u_h| \leq N 
 (F_{m}+G_{m}) ,
\end{equation}
where 
$$
 F_{n}
=\sum_{k\leq n} \sup_{H_{T}}
|D^{k}f_h| , 
\quad 
G_n=\sum_{k\leq n} \sup_{\bR^{d}}
|D^{k}g_h| ,
$$
  and $N$ 
depends only on   $\tau_{0}$,
$m$, $\delta$,  $c_0$, 
$K$, 
 $|\Lambda_1|$, $\|\Lambda_1\|$, $M_{0},...,M_{m}$
($N$ depends on fewer parameters if $m\leq1$). 
\end{theorem}

To proceed further we need a few formulas. 

\begin{lemma}                   \label{lemma 17.27.01.08}
Let $\varphi$ be a function on $H_T$ and
$n\geq0$ be an integer.  

(i) Assume that  
the derivatives of 
$\varphi$  in $x\in\bR^d$ 
up to   order $n+1$ are continuous functions in $x$. 
Then for each $h>0$ 
\begin{equation}
                                         \label{08.2.19.3}
D^{n}_{h}
\sum_{\lambda\in\Lambda_1}p_{\lambda}\delta_{h,\lambda}\varphi
=\sum_{\lambda\in\Lambda_1}p_{\lambda}
\int_0^1\theta^{n}\partial_{\lambda}^{n+1}
\varphi(t,x+h\theta\lambda)\,d\theta 
\end{equation}
on $H_T$, where 
 $
\partial_{\lambda}\varphi
$ 
is introduced in \eqref{08.12.16.2}. 

(ii) Assume that the derivatives of 
$\varphi$ in $x$  up to 
  order  $n+2$  
are continuous functions in $x$, and that 
Assumption \ref{assumption 14.23.01.08} 
holds. Then 
\begin{equation}                                \label{08.2.19.4}
D^{n}_{h}
\sum_{\lambda\in\Lambda_1}h^{-1}q_{\lambda}
\delta_{h,\lambda}\varphi
=\sum_{\lambda\in\Lambda_1}q_{\lambda}
\int_0^1(1-\theta)\theta^{n}\partial_{\lambda}^{n+2}
\varphi(t,x+h\theta\lambda)\,d\theta, 
\end{equation}
on $H_T$. 
\end{lemma}
\begin{proof} 
By  Taylor's formula applied
to $\varphi(t,x+h\theta\lambda)$ as a
function of $\theta\in
[0,1]$
$$
\delta_{h,\lambda}\varphi(t,x)=
\int_0^1\partial_{\lambda}\varphi(t,x+h\theta\lambda)\,d\theta
$$
and  
$$
\delta_{h,\lambda}\varphi(t,x)
=\partial_{\lambda}\varphi(t,x)+
h\int_0^1(1-\theta)
\partial_{\lambda}^2\varphi(t,x+h\theta\lambda)\,d\theta. 
$$ 
Multiplying the first equality by $p_{\lambda}$ 
and summing up in $\lambda$ over $\Lambda_1$ we obtain
\eqref{08.2.19.3} for $n=0$. 
Multiplying the second equality by $q_{\lambda}$,
 summing up in $\lambda$ over $\Lambda_1$
 we obtain
\eqref{08.2.19.4} for $n=0$
since 
$$
\sum_{\lambda\in\Lambda_1}
q_{\lambda}\partial_{\lambda}\varphi=0
$$
due to Assumption \ref{assumption 14.23.01.08}.

After that it only remains to differentiate $n$ times
in $h$
both parts of the particular case of formulas
 \eqref{08.2.19.3} and  \eqref{08.2.19.4}.
The lemma is proved.
\end{proof}

Introduce
$$
u^{(j)}_{h}=D^{j}_{h}u_{h}
$$
and observe that by Remark \ref{remark 08.2.19.1}
under Assumption \ref{assumption 16.12.07.06} the functions
$\partial^{n}_{\lambda}u^{j}$ are well defined
if $n+j\leq m$. 
 By combining this with Lemma \ref{lemma 17.27.01.08}
 and the Leibnitz formula
we obtain the following.
\begin{corollary}
                                  \label{corollary 08.2.19.1}
Let Assumptions \ref{assumption 16.12.07.06}
and \ref{assumption 14.23.01.08} be satisfied. 
Let  $k \geq1$ be an integer such that $k+2\leq m$.
Then
\begin{equation}
                                         \label{08.2.19.5}                                   
  u^{(k)}_{h}(t,x)=  
 \int_{0}^{t}\big(L_{h}  u^{(k)}_{h}(s,x)+R^{k}_{h}(s,x)
 \big)\,ds
\end{equation}
on $(0,h_{0}]\times H_{T}$,
where
$$
R^{k}_{h}(t,x)=\sum_{i=1}^{k}C^{i}_{k}
\sum_{\lambda\in\Lambda_{1}}\int_{0}^{1}\theta^{i}\big[
 p_{\lambda}(t,x)(\partial_{\lambda}^{i+1}u_{h}^{(k-i)})
(t,x+h\theta
\lambda)
$$
$$
+(1-\theta)q_{\lambda}(t,x) (
\partial_{\lambda}^{i+2}u_{h}^{(k-i)})
(t,x+h\theta
\lambda)\big]\,d\theta.
$$
\end{corollary}

Now we are ready to prove Theorems \ref{theorem 17.25.06.07} 
and \ref{theorem 08.2.19.1}.  

\bigskip
{\it Proof of Theorem \ref{theorem 17.25.06.07}}. 
If $m=2$ or $k=0$, our assertion follow directly from
Theorem \ref{theorem 4.7.2}. Therefore, in the rest of
the proof we assume that $m\geq3$ and $k\geq1$.

We will be using \eqref{08.2.19.5}. Observe that
if   $1\leq i\leq k$, then 
$$
(i+2)+r+(k-i)=k+2+r\leq 3k+r\leq m.
$$
Thus by Remark \ref{remark 08.2.19.1} we know that
$D^{i+2+r}u^{(k-i)}_{h}$ are bounded and continuous
on $H_{T}$. It follows that $R^{k}_{h}\in\frB^{  r }$.
 By  Theorem \ref{theorem 4.7.2}
with $r$ in place of $m$  we obtain
$$
I_{kr}:=\sup_{H_{T}}\sum_{j\leq r}|D^{j} u^{(k)}_h | 
\leq N
\sup_{H_{T}}\sum_{j\leq r}  |D^{j}R^{k}_{h} |. 
$$
It is not hard to see that 
$$
|D^{j}R^{k}_{h} |\leq N\sup_{H_{T}}
\sum_{i=1}^{k}\sum_{n=1}^{i+2+j}|D^{n}u_{h}^{(k-i)}| 
\leq N\sum_{i=1}^{k}I_{k-i,i+2+j}.
$$
Hence,
$$
I_{kr}\leq  N 
\sum_{i=1}^{k}I_{k-i,i+2+r}.
$$
Here on the right the first index of $I_{kr}$ is reduced
by at least  1   and the sum of indices
increased by 2. Therefore,   after
$k$ iterations we will come
to the inequality
$$
I_{kr}\leq NI_{0, k+2k+r}.
$$
 
It only remains to
observe that $I_{0,3k+r}\leq I_{0,m}$ and the latter
quantity is estimated in Theorem \ref{theorem 4.7.2}.
The theorem is proved.

\bigskip
{\it Proof of Theorem \ref{theorem 08.2.19.1}}.
First of all observe that the symmetry assumption
and
\eqref{08.11.7.5} imply that for any smooth function
$\varphi(x)$, odd $i\geq0$, and any multi-index
$\alpha$, such that $|\alpha|\leq m$, we have
\begin{equation}
                                               \label{08.2.20.2}
\sum_{\lambda\in\Lambda_{1}}(D^{\alpha}p_{\lambda})
\partial_{\lambda}^{i+1}\varphi
=\sum_{\lambda\in\Lambda_{1}}(D^{\alpha}q_{\lambda})
\partial_{\lambda}^{i+2}\varphi=0.
\end{equation}
If $k=1$ and an integer $n\leq r$, 
then owing to \eqref{08.2.20.2}
 $$
\big|D^{n}\sum_{\lambda\in\Lambda_{1}}q_{\lambda}(t,x)
(\partial_{\lambda}^{3}u_{h})(t,x+h\theta\lambda)\big|
$$
$$
=\big|D^{n}\sum_{\lambda\in\Lambda_{1}}q_{\lambda}(t,x)
\big[(\partial_{\lambda}^{3}u_{h})(t,x+h\theta\lambda)
-\partial_{\lambda}^{3}u_{h}(t,x)\big]\big|
$$
$$
\leq Nh\sup_{H_{T}}\sum_{i\leq r}|D^{ i+4}u_{h}|
\leq Nh\|u\|_{m}
\leq N (\|f\|_m+\|g\|_m) h=:NJh,
$$
where the last two estimates follow from the fact 
that $r+4=r+3k+1\leq m$ and from Theorem
\ref{theorem 17.25.06.07}, respectively.
Similarly,
$$
\big|D^{n}\sum_{\lambda\in\Lambda_{1}}p_{\lambda}(t,x)
(\partial_{\lambda}^{2}u_{h})(t,x+h\theta\lambda)\big|
$$
$$
=\big|D^{n}\sum_{\lambda\in\Lambda_{1}}p_{\lambda}(t,x)
\big[(\partial_{\lambda}^{2}u_{h})(t,x+h\theta\lambda)
-\partial_{\lambda}^{2}u_{h}(t,x)\big]\big|
\leq NJh.
$$
Hence,
$$
\sup_{H_{T}}\sum_{n\leq r}|D^{n}R^{1}_{h}|
\leq N (\|f\|_m+\|g\|_m) h\leq NJh
$$
and applying Theorem \ref{theorem 4.7.2}
to \eqref{08.2.19.5}  
 yields
\eqref{08.2.19.7}.

Now we proceed by induction on $k$. Assume that
for an odd number   $j$ estimate \eqref{08.2.19.7}
holds whenever $3k+r\leq m-1$ {\em and\/} odd $k\leq j$.
This hypothesis is justified by the above for $j=1$
and  to prove the theorem it suffices to show that
the hypothesis  also holds with $ j+2$ in place of $j$.
Take an odd $k$ and an integer $r$ such that
$$
k\leq j+2,\quad 3k+r\leq m-1
$$
and again use \eqref{08.2.19.5}. 
As above, to obtain \eqref{08.2.19.7} 
it suffices to prove
that
\begin{equation}
                                               \label{08.2.20.4}
\sup_{H_{T}}\sum_{n\leq r}|D^{n}R^{k}_{h}|
\leq  NJh.
\end{equation}
Take an integer $n\leq r$.
Observe that  if $1\leq i\leq k$ and $i$ is even, then
$k-i$ is odd and $k-i\leq j+2-i\leq j$ and
$$
3(k-i)+i+2+n=3k+n-2i+2  \leq m-1-2i+2\leq m-1
$$ 
so that by the induction hypothesis
\begin{equation}                              \label{08.2.20.6}
\sup_{H_{T}}\big|D^{n}\sum_{\lambda\in\Lambda_{1}} 
 q_{\lambda}(t,x) (
\partial_{\lambda}^{i+2} u_{h}^{(k-i)} )
(t,x+h\theta
\lambda) \big|
\leq NJh.
\end{equation}
If $1\leq i\leq k$ and $i$ is odd,
then $i+2$ is odd too and as in the beginning of the
proof
$$
\big|D^{n}\sum_{\lambda\in\Lambda_{1}} 
 q_{\lambda}(t,x) (
\partial_{\lambda}^{i+2} u_{h}^{(k-i)})
(t,x+h\theta
\lambda) \big|
$$
$$
=\big|D^{n}\sum_{\lambda\in\Lambda_{1}}  
q_{\lambda}(t,x) \big[(
\partial_{\lambda}^{i+2} u_{h}^{(k-i)})
(t,x+h\theta
\lambda)-\partial_{\lambda}^{i+2} u_{h}^{(k-i)}(t,x)
 \big]\big|
$$
$$
\leq Nh\sup_{H_{T}}\sum_{i\leq k,l\leq r }
|D^{l+i+3} u_{h}^{(k-i)}|,
$$
where the last sup is majorated by $NJ $
owing to Theorem \ref{theorem 17.25.06.07} since
$$
3(k-i)+r+i+3  \leq m-1-2i+3\leq m.
$$
In both situations we have \eqref{08.2.20.6}.
Similarly, if $1\leq i\leq k$ and $i$ is odd, then   
$i+1$ is even and
$$
\big|D^{n}\sum_{\lambda\in\Lambda_{1}} 
 p_{\lambda}(t,x) (
\partial_{\lambda}^{i+1} u_{h}^{(k-i)})
(t,x+h\theta
\lambda) \big|
$$
$$
=\big|D^{n}\sum_{\lambda\in\Lambda_{1}} 
p_{\lambda}(t,x) \big[(
\partial_{\lambda}^{i+1} u_{h}^{(k-i)})
(t,x+h\theta
\lambda)-\partial_{\lambda}^{i+1} u_{h}^{(k-i)}(t,x)\big]\big|
$$
$$
\leq Nh\sup_{H_{T}}\sum_{i\leq k,l\leq  r }
|D^{l+i+2} u_{h}^{(k-i)}|,
$$ 
where the last sup is majorated by $NJ $ again
owing to Theorem \ref{theorem 17.25.06.07} since
$$
3(k-i)+r+i+2  \leq m-1-2i+2\leq m.
$$
Finally, if $1\leq i\leq k$ and $i$ is even, then
$k-i$ is  odd, 
  $k-i\leq  j+2-i\leq j$, and
$$
3(k-i)+r+i+1 \leq m-1-2i+1< m-1,
$$
so that by the induction hypothesis
$$
\big|D^{n}\sum_{\lambda\in\Lambda_{1}} 
 p_{\lambda}(t,x) (
\partial_{\lambda}^{i+1} u_{h}^{(k-i)})
(t,x+h\theta
\lambda) \big|
\leq NJ h,
$$
which is now shown to hold in both subcases.
By combining this with \eqref{08.2.20.6}
we come to \eqref{08.2.20.4} and the theorem is
proved.

\mysection{Proof of Theorems \protect\ref{theorem 08.11.2.1},
\protect\ref{theorem 08.11.7.2}, and
\protect\ref{theorem 08.11.3.1}}
                                       \label{section 1.25.11.08}

\begin{proof}[Proof of Theorem \ref{theorem 08.11.2.1}]
First we replace $q_{\lambda}$ with symmetric ones
using the fact that the symmetrization
 does not affect formula \eqref{08.11.3.6}. 
To this end
introduce 
$$
\Lambda_{1}^{s}=\Lambda_{1}
\cap(-\Lambda_{1}),\quad
 \hat{\Lambda}_{1}=\Lambda_{1} \cup(-
\Lambda_{1}).
$$
 On $\Lambda_{1}^{s}$
 we set $\hat{q}_{\pm\lambda}
=(1/2)(q_{\lambda}+q_{-\lambda})$. If $\lambda
\in\pm(\Lambda_{1}\setminus\Lambda_{1}^{s})$ we set
$\hat{q}_{\lambda}=(1/2) q_{\pm\lambda}$. Then
$\hat{\Lambda}_{1}$ and $\hat{q}_{\lambda}$
satisfy the symmetry condition (S) and can be used
to represent the first term on the right in
\eqref{08.11.3.6} in place of the original ones. 
Next, we  redefine and extend $p_{\lambda}$  
introducing $\hat{p}_{\lambda}$
on $\hat{\Lambda}_{1}$, so that
$\hat{p}_{\lambda}=M_{0} +p_{\lambda}$ 
on $\Lambda_{1}^{s}$,
 for $\lambda
\in \Lambda_{1}\setminus\Lambda_{1}^{s} $ we set
 $\hat{p}_{\pm\lambda}=M_{0} \pm(1/2)p_{\lambda}$,
and for $-\lambda
\in \Lambda_{1}\setminus\Lambda_{1}^{s} $ we set
 $\hat{p}_{\pm\lambda}=M_{0} \mp(1/2)p_{-\lambda}$. 
 (Remember that for the constant $M_0$ 
from Assumption \ref{assumption 16.12.07.06} 
we have $|p_{\lambda}|\leq M_0$.) 
Then $\hat{\Lambda}_{1}$ and $\hat{p}_{\lambda}$
can be used to represent the second term 
on the right in
\eqref{08.11.3.6} in place of the original ones. 
One of  the  advantages 
of the new $\hat{p}_{\lambda}$ is that
 $\hat{p}_{\lambda}\geq0$, which implies that the new
$\chi_{\lambda}$ satisfies Assumption
\ref{assumption 1.26.11.06}.

Define $\tau_{\lambda}>0$ arbitrarily.
As in Remark 6.4 of \cite{GK1} and Remark 4.3 of \cite{GK2} 
one shows that 
Assumptions \ref{assumption 11.22.11.06}
and \ref{assumption 07.10.16.1} are also 
satisfied for any $\delta\in(0,1)$, say $\delta=1/2$, 
if $c$
is sufficiently large (independently of $h$) and
$\tau_0  >0 ,K$,  and $\cK$
are chosen appropriately and depending only on
$d,|\Lambda_{1}|,
 \|\Lambda_1\| , M_{0},M_{1},M_{2}$.
We first concentrate on the case that $c$ is indeed 
sufficiently large. In that case by Theorem 
\ref{theorem 17.25.06.07}, 
for $h\in(0,h_{0}]$, there exists
a unique solution $u_{h}(t,x)$ of class $\frB^{m}_{T}$
 satisfying equation
\eqref{equation} with $\hat{L}_{h}$ in place of
$L_{h}$, where $\hat{L}_{h}$ is constructed from
$\hat \Lambda_{1}$, $\hat q_{\lambda}$, and 
 $\hat p_{\lambda}$.
Furthermore,  
\begin{equation}
                                         \label{08.11.6.1}
\|u_{h}\|_{m}
\leq N(\|f \|_{m}+\|g \|_{m}),
\end{equation}
 where 
$N$ is a constant depending only on  
$m, \inf c$, 
$|\Lambda_1|$, 
$M_0$,..., $M_{m}$, and
 $\|\Lambda_1\| $. Upon observing that owing 
to Remark \ref{remark 08.11.6.1}
$$
|\hat{L}_{h}u_{h} |\leq N(
\sup_{H_{T}}|D^{2}u_{h} |+\sup_{H_{T}}|D u_{h} |+
\sup_{H_{T}}| u_{h} |) 
$$
with $N$ independent of $h$,
we conclude from the equation for $u_{h}$
that their first derivatives in $t$ are 
bounded uniformly in $h$.
Therefore, there exists
a sequence $h(n)\downarrow0$ such that
$u_{h(n)}$ converges uniformly on $ [0,T]\times
\{x:|x|\leq R\} $ for any $R$ to a 
continuous function $v$.  Then
\eqref{08.11.6.1} implies that $v\in \frB^{m}$
and 
\begin{equation}
                                         \label{08.11.6.3} 
\|v\|_{m}\leq N(\|f \|_{m}+\|g \|_{m}) 
\end{equation}
with the same $N$ as in \eqref{08.11.6.1}.
If we take $\tau_{\lambda}\equiv1$, then
Remark 6.4 of \cite{GK1} and Remark 4.3 of \cite{GK2}
  imply that
 both $N$'s can be chosen to
depend only on $d$, $m$, $\inf c$,
$|\Lambda_{1}|$, and
$M_{0},...,M_{m}$.

Next, the modified equation \eqref{equation}
yields that for any $\phi\in C^{\infty}_{0}(\bR^{d})$
and $t\in[0,T]$
$$
\int_{\bR^{d}}u_{h}(t,x)\phi(x)\,dx
=\int_{\bR^{d}}g( x)\phi(x)\,dx
$$
$$
+\int_{0}^{t} \int_{\bR^{d}}\sum_{\lambda\in\hat{\Lambda}_{1}}
u_{h}(s,x)[(1/2)\Delta_{h,\lambda}(\hat{q}_{\lambda}
\phi)+\delta_{h,-\lambda}(
 \hat{p}_{\lambda }\phi)](s,x)
 \,dxds
$$
$$
+\int_{0}^{t} \int_{\bR^{d}}(-cu_{h}+f)\phi (s,x)
\,dx\,ds. 
$$
We pass to the limit in this equation and find that
$v$ satisfies an integral equation, integrating 
by parts in which proves that $v$ is a solution of
\eqref{08.11.2.3}.

Finally, we notice that the case that $c$ is not large
is reduced to the above one
by usual change of the unknown function taking
$v(t,x)e^{\lambda t}$ 
in place of $v$ for an appropriate $\lambda$, which
leads to  subtracting $\lambda v$ from 
the right-hand side of \eqref{08.11.2.1}. 
For the new equation we then find
a solution admitting estimate
\eqref{08.11.6.3} with $N$ independent of $T$
but coming back to the solution of the original equation
will bring an exponential factor depending on $T$.

This and uniqueness proved in Section \ref{section 08.12.17.1}
finish proving the theorem.
\end{proof}

\begin{remark}
                                     \label{remark 08.11.6.4}
In the above proof we considered arbitrary $\tau_{\lambda}>0$
for the following reason. If
Assumptions 
\ref{assumption 16.12.07.06}
through \ref{assumption 07.10.16.1}
hold with $m\geq 2$, then by Theorem \ref{theorem 17.25.06.07}
estimate
\eqref{08.11.6.1}
and hence   \eqref{08.11.6.3} hold with $N$
depending only on 
$m, \delta, c_0,\tau_0,K $, 
$M_{ 0 }$, 
..., $M_{m}$, $|\Lambda_{1}|$, and
$\|\Lambda_1\| $.
This proves
  the assertion of Theorem \ref{theorem 08.11.3.1} 
regarding the constant $N$ in Theorem \ref{theorem 08.11.2.1}.

\end{remark}

\begin{proof}[Proof of Theorem \ref{theorem 08.11.7.2}]
Notice that for each $j= 1,\dots,k$   equation 
\eqref{08.11.2.6} 
  does not involve the unknown functions 
$u^{(l)}$ with indices $l>j$. Therefore we can solve 
\eqref{08.11.2.6}  and 
prove the statements (i) and (ii) recursively on 
$j$. 

First we prove that there is at most 
one solution $(u^{(1)},\dots,u^{(k)})$ in the space 
$\frB^{2}\times\dots\times\frB^{2}$. 
Denote
$$
S_{j}=\sum_{i=1}^{j}C_{j}^{i}\cL^{(i)}u^{(j-i)}.
$$
We may assume that $u^{(0)}=0$. Then 
clearly $S_1=0$ and by Theorem \ref{theorem 08.11.2.1} 
we have $u^{(1)}=0$. If for a $j\in\{2,\dots k\}$ 
we have $u^{(1)}=u^{(2)}=\dots =u^{(j-1)}=0$, then 
clearly $S_j=0$ 
which by Theorem \ref{theorem 08.11.2.1} 
yields $u^{(j)}=0$. Hence the statements on uniqueness follow 
because for every $j=1,2,\dots, k$ we obviously have  
$\frB^{m-3j}\subset \frB^{2}$ when $m\geq 3k+2$ 
and $\frB^{m-2j}\subset \frB^{2}$ when $m\geq 2k+2$.  

While dealing with the existence of a  
solution
first take $j=1$. 
Observe that by Theorem \ref{theorem 08.11.2.1} 
we have $u^{(0)}\in\frB^m$ with $m\geq 5$ in case (i) 
and with $m\geq4$ in case (ii). Thus in case (i) 
we have $S_1\in\frB^{m-3}\subset\frB^{2}$ and by 
Theorem
\ref{theorem 08.11.2.1} it follows that
there exists $u^{(1)}\in\frB^{m-3}$ satisfying
\eqref{08.11.2.6} and admitting the estimate
$$
\|u^{(1)}\|_{m-3}
\leq N\|u^{(0)}\|_{m}.
$$
Taking the estimate of the last term again from Theorem   
\ref{theorem 08.11.2.1} we obtain \eqref{08.11.7.3}
for $j=1$.
In  case  (ii)   we have actually 
better smoothness of
$S_{1}$, because the first sum in \eqref{08.11.2.5}
is zero for $i=1$ and, for that matter, for all odd $i$.
It follows that $S_{1}\in\frB^{m-2}$ and this leads to
\eqref{08.11.7.4}  for $j=1$ as above.
By adding that under the conditions (S) and
 \eqref{08.11.7.5} we have
$\cL^{(1)}=0$, $S_{1}=0$, and $u^{(1)}=0$, we obtain 
  \eqref{08.11.7.7}
for $j=1$.

Passing to higher $j$ we assume that $k\geq2$.   
Suppose that,
for a $j\in\{2,...,k \} $ 
we have found $u^{(1)}$,...,$u^{(j-1)}$
with the asserted properties.  Then in the case (i) we have 
$$
\cL^{(i)}u^{(j -i)}\in \frB^{m-3j} \subset\frB^{2} 
$$ 
for $i=1,\dots,j $, since 
$$
 m-3(j -i)-(i+2)=m-3j+2i-2\geq m-3j\geq 2.
$$
Hence $S_{j}\in \frB^{m-3j}$ 
and therefore by Theorem
\ref{theorem 08.11.2.1}  
there exists $u^{(j)}\in\frB^{m-3j}$ satisfying
\eqref{08.11.2.6} and admitting the estimate 
$$
\|u^{(j)}\|_{m-3j}
\leq N\sum_{i=1}^{j }
\|u^{(j-i)}\|_{m-3 j+3 }
$$
$$
\leq N\sum_{i=1}^{j }
\|u^{(j-i)}\|_{m-3 (j-i) }
\leq N(\|f\|_{m}+\|g\|_{m}),
$$
where the last inequality follows by the induction
hypothesis.

In case  (ii)   we 
take into account that due to condition (S) 
 we have 
\begin{equation}                              \label{1.25.08.08}
\sum_{\lambda\in\Lambda_{1}}q_{\lambda}
\partial_{\lambda}^{i+2}\varphi=0, 
\end{equation}                                
and due to condition \eqref{08.11.7.5} we have 
\begin{equation}                              \label{2.25.08.08}
\sum_{\lambda\in\Lambda_{1}}p_{\lambda}
\partial_{\lambda}^{i+1}\varphi=0
\end{equation}
for odd numbers $i$ and sufficiently smooth 
functions $\varphi$. 
It follows that in case  (ii)    for
$i=1,...,j$ we have
$$
\cL^{(i)}u^{(j-i)}\in \frB^{m-2j}
\subset\frB^{2} , 
$$ 
since $\cL^{(1)}u^{(j-1)}\in \frB^{m-2(j-1)-2}$ and for $i\geq2$
$$
m-2(j-i)-(i+2)=m-2j+i-2
\geq m-2j\geq2.
$$
Hence $S_{j}\in \frB^{m-2j}$ 
and therefore by Theorem
\ref{theorem 08.11.2.1}  
there exists 
$u^{(j)}\in\frB^{m-2j}$ satisfying
\eqref{08.11.2.6} and admitting the estimate 
$$
\|u^{(j)}\|_{m-2j}
\leq N\|u^{(j-1)}\|_{m-2 j+2}
+N\sum_{i=2}^{j}\|u^{(j-i)}\|_{m-2j+3}
$$
$$
\leq N\sum_{i=1}^{j}\|u^{(j-i)}\|_{m-2(j-i)}, 
$$ 
and by using the induction hypothesis we come
to \eqref{08.11.7.4}.   

Furthermore,
in case (ii) if \eqref{08.11.7.5} is satisfied,
our induction hypothesis says that
$u^{(l)}=0$ for all odd $l\leq j-1$.
If $j$ is even, then, obviously, $u^{(l)}=0$ 
for all odd $l\leq  j$ as well. If $j$ is odd then 
to carry the induction forward it only remains
to prove that $u^{(j)}=0$. However,
for odd $i$ we have 
$$
\cL^{(i)}u^{(j-i)}=0
$$
  due to \eqref{1.25.08.08}-\eqref{2.25.08.08}.
This equality also holds
if $i\geq2$ and $i$ is even, since then $j-i$
is odd and $u^{(j-i)}=0$ by assumption. Thus, $S_{j}=0$ and $u^{(j)}=0$.
\end{proof}

\begin{remark}
                                     \label{remark 08.11.8.1}
The above proof is based on 
Theorem \ref{theorem 08.11.2.1}
and leads to estimates \eqref{08.11.7.3} and \eqref{08.11.7.4}
with $N$ depending only on the same parameters as
in Theorem \ref{theorem 08.11.2.1}.
Therefore, according to Remark \ref{remark 08.11.6.4}
 if
Assumptions 
\ref{assumption 16.12.07.06}
through \ref{assumption 07.10.16.1}
are satisfied and the restrictions on $m$ and $k$
from Theorem \ref{theorem 08.11.7.2} are met,
then  the constants $N$ in
estimates \eqref{08.11.7.3} and \eqref{08.11.7.4}
depend  only on 
$m, \delta, c_0,\tau_0,K $, 
$M_0$, ..., $M_{m}$, $|\Lambda_{1}|$, and
$\|\Lambda_1\| $.
This proves the
 part of assertions of Theorem \ref{theorem 08.11.3.1}
concerning Theorem \ref{theorem 08.11.7.2}. The proof of
its remaining assertions can be obtained in the same way
and is left to the reader.

\end{remark}

 \mysection{Proof of Theorem \protect\ref{theorem 6.25.08.08} 
 and \protect\ref{theorem 1.25.11.08}} 
                                     \label{section 5.23.11.08}

We need  some lemmas.   
 The first one is a simple lemma from 
undergraduate calculus on Taylor's 
expansion. 
\begin{lemma}                                       \label{lemma Taylor}  
Let $F$ be 
a real-valued function on $(0,1]$ 
such that for an integer $m\geq 0$  
the derivative $F^{(m+1)}(h)$ of order $m+1$ exists   
for all $h\in (0,1]$,  and $F^{(m+1)}$ is a 
bounded function on $(0,1]$.  
Then 
\begin{equation*}                                         \label{15.15.07.06} 
F^{(k)}(0):=\lim_{s\downarrow0}F^{(k)}(s)  
\end{equation*}
exist for $0\leq k\leq m$, and 
\begin{equation*}                                          \label{16.15.07.06}
F(h)=\sum_{k=0}^{m}\frac{h^k}{k!}F^{(k)}(0)+R_m(h)
\end{equation*}
holds for $h\in[0,1]$ with 
$$
R_m(h)=\int_0^h\frac{(h-s)^{m}}{m!}F^{(m+1)}(s)\,ds,
$$
so that 
$$
|R_m(h)|\leq \sup_{s\in(0,1]}|F^{(m+1)}(s)|
\frac{h^{m+1}}{(m+1)!}\quad 
\text{\rm for all $h\in[0,1]$}. 
$$
\end{lemma}

To formulate our next lemma we recall the operators 
$L_h$, $\cL$ and $\cL^{(i)}$, defined 
in \eqref{1.24.11.08}, \eqref{08.11.3.6}, and 
\eqref{08.11.2.5}, respectively, and 
for 
each $h\in(0,h_0]$ and integer $j\geq 0$ 
introduce  the operator 
$$
\mathcal O^{(j)}_h=L_h-\cL
-\sum_{1\leq i\leq j}\frac{ h^i}{i!}\cL^{(i)}.
$$

\begin{lemma}                           \label{lemma 3.25.08.08}
Let Assumption \ref{assumption 14.23.01.08} hold. 
Assume that for some integer $l\geq0$
the functions $p_{\lambda},q_{\lambda}$ 
belong to $\frB^l$ for all $\lambda\in\Lambda_1$. 
Then for  any integer $j\geq0$  
\begin{equation}                         \label{6.23.11.08}
\|\mathcal O^{(j)}_h\varphi\|_l\leq N
\|\varphi\|_{l+j+3}h^{j+1}
\end{equation}
for all $h\in(0,h_0]$ and  
$\varphi\in\frB^{l+j+3}$, where $N$ is a constant 
depending only on $ |\Lambda_1|, M_{0},...,M_{l}$.  
\end{lemma}
 
\begin{proof}
 We may assume that 
the derivatives in $x$ of $\varphi$ 
up to order $l+j+3$ 
are bounded continuous functions on $H_T$.
By Lemma \ref{lemma 17.27.01.08}  the derivatives 
of the function $L_h\varphi$ in $h$ up to  
the
$(l+j+1)$st order 
are bounded functions on $(0,h_0]\times H_T$ 
and        
\begin{equation*}                             \label{4.25.08.08}           
(\cL\phi)(t,x)=\lim_{h\to0}(L_h\varphi)(t,x),
\end{equation*} 
\begin{equation*}                             \label{5.25.08.08}  
(\cL^{(i)}\phi)(t,x)
=\lim_{h\to0}(D^i_hL_h\phi)(t,x). 
\end{equation*}
Thus applying Lemma \ref{lemma Taylor} 
to $F(h):=L_h\varphi(t,x)$ for 
fixed $(t,x)$ and using 
Lemma \ref{lemma 17.27.01.08}, we have 
$$
\mathcal O_h^{(j)}\varphi
=\int_0^h\frac{(h-\vartheta)^{j}}{j!}
L^{(j+1)}_{\vartheta}\varphi\,d\vartheta  
$$
$$
=\sum_{\lambda\in\Lambda_1}q_{\lambda}
\int_0^h\frac{(h-\vartheta)^{j}}{j!}\int_0^1
(1-\theta)\theta^{j+1}\partial_{\lambda}^{j+3}
\varphi(t,x+\vartheta\theta\lambda)\,d\theta
\,d\vartheta
$$
$$
+\sum_{\lambda\in\Lambda_1}p_{\lambda}
\int_0^h\frac{(h-\vartheta)^{j}}{j!}\int_0^1
\theta^{j+1}\partial_{\lambda}^{j+2}
\varphi(t,x+\vartheta\theta\lambda)\,d\theta
\,d\vartheta. 
$$
 Now estimate 
\eqref{6.23.11.08} follows easily. 
 \end{proof}

 The next lemma is a version 
of the maximum principle for $\partial/\partial t-L_{h}$. It
is a special  case of Corollary 3.2 in \cite{GK1}.

\begin{lemma}
                                        \label{lemma 3.6.1}
Let Assumption
\ref{assumption 16.12.07.06} with $m=0$ be satisfied and let
$\chi_{h,\lambda}\geq0$ for all $\lambda\in\Lambda_{1}$.
Let $v$ be a bounded
 function on $H_{T}$, such that 
the partial derivative $\partial v(t,x)/\partial t$
exists in $H_{T}$. 
Let   
$F$ be a nonnegative integrable function on $[0,T]$,
and let $C$ be a nonnegative bounded function on $H_T$ 
such that 
$$
\nu:=\sup_{H_{T}}(C-c)<0. 
$$ 
Assume that
  for all $(t,x)\in H_{T}$ we have
\begin{equation}
                                               \label{07.9.23.1}
 \frac{\partial}{\partial t}  v 
\leq  L_{h} v + C\bar{v}_{+}+F,
\end{equation}
where $\bar{v} (t)=\sup\{v (t,x):x\in \bR^d\}$.
Then in $[0,T]$ we have
\begin{equation}
                                        \label{3.11.1}
\bar{v}(t)\leq  
\bar v_{+}(0)
+|\nu|^{-1}\sup_{[0,t]}F,
\end{equation}
where $a_+:=(|a|+a)/2$ for real numbers $a$. 
\end{lemma}

\begin{proof}[Proof of Theorem \ref{theorem 6.25.08.08}] 
By taking  $u_he^{-(M_0+1)t} $ in place of $u_h$, 
we may assume that $c\geq1$. 
Consider first the case $k=0$. 
Since $m\geq3$, by Theorem \ref{theorem 08.11.2.1} 
equation \eqref{08.11.3.6} has a 
solution $u^{(0)}$, which belongs to $\frB^m$ and 
estimate \eqref{08.11.3.2} holds. 
Clearly, $w:=u_h-u^{(0)}$ is the unique bounded solution 
of the equation  
\begin{equation}                                \label{6.25.08.08}
w(t,x)=\int_0^t\big( L_hw(s,x)+F(s,x)\big)\,ds, 
\quad (t,x)\in H_T, 
\end{equation}
 where 
$F:=\mathcal O_h^{(0)}u^{(0)}=L_hu^{(0)}-\cL u^{(0)}$. 
By Lemma \ref{lemma 3.25.08.08} and 
estimate \eqref{08.11.3.2} 
$$
\|\mathcal O^{(0)}_hu^{(0)}\|_0\leq N
\sum_{\lambda\in\Lambda_1}
(\|p_{\lambda}\|_0+\|q_{\lambda}\|_0)\|u^{(0)}\|_{3}h
\leq N(\|f\|_3+\|g\|_3)h
$$
with constants $N$ depending only on $d$, $|\Lambda_1|$ 
$M_0, M_1, M_3$, and $T$. After that
  an application of Lemma \ref{lemma 3.6.1} to 
equation \eqref{6.25.08.08} proves the statement 
of Theorem \ref{theorem 6.25.08.08} for $k=0$.

Let $k\geq1$. Then by Theorem \ref{theorem 08.11.7.2} 
the system of equations \eqref{08.11.2.6} has a 
bounded solution $\{u^{(i)}\}_{i=1}^k$. 
Observe that for  
\begin{equation}
                                              \label{08.12.16.3}
w:=u_h-\sum_{j=0}^{k}u^{(j)}\frac{h^j}{j!}  
\end{equation}
we have 
equation \eqref{6.25.08.08}
with
$$
F:=L_hu^{(0)}-\cL u^{(0)}
+\sum_{j=1}^{k}L_hu^{(j)}\frac{h^j}{j!}
-\sum_{j=1}^{k}\cL u^{(j)}\frac{h^j}{j!}-G, 
$$
  and
$$
G:=\sum_{j=1}^k\sum_{i=1}^{j}
\frac{1}{i!(j-i)!}\cL^{(i)}u^{(j-i)}h^j
=\sum_{i=1}^k\sum_{j=i}^{k}
\frac{1}{i!(j-i)!}\cL^{(i)}u^{(j-i)}h^j 
$$
$$
=\sum_{i=1}^k\sum_{l=0}^{k-i}
\frac{1}{i!l!}\cL^{(i)}u^{(l)}h^{l+i} 
=\sum_{l=0}^{k-1}\frac{h^{l}}{l!}
\sum_{i=1}^{k-l}\frac{h^i}{i!}\cL^{(i)}u^{(l)}
$$
$$
=\sum_{j=0}^{k}\frac{h^{j}}{j!}
\sum_{1\leq i\leq k-j} \frac{h^i}{i!}\cL^{(i)}u^{(j)}.
$$
Hence by simple arithmetics 
\begin{equation}
                                              \label{08.12.16.4}
F=\sum_{j=0}^k\frac{h^j}{j!} 
\mathcal O_h^{(k-j)}u^{(j)} .  
\end{equation}
Notice that 
$$
\text{$k-j+3\leq m-3j$ \,for $j=0,1,\dots,k$ \,in case (i)},  
$$
$$
\text{$k-j+3\leq m-2j$ \,for $j=0,1,\dots,k$ \,in case (ii)},  
$$
$$
\text{$k-j+3\leq m-2j$ \,for $j=0,1,\dots,k-1$ \,in case (iii)}.   
$$
Therefore by Theorem \ref{theorem 08.11.7.2} 
under each of (i), (ii), and (iii)
$$
\|u^{(j)}\|_{k-j+3}
\leq N(\|f\|_m+\|g\|_m)
$$
for $j=0,1\dots,k$ ($u^{(k)}=0$ in the case  (iii)).  
Thus by Lemma \ref{lemma 3.25.08.08} 
$$
 \|\mathcal O_h^{(k-j)}u^{(j)}\|_0 \leq 
Nh^{k-j+1} \|u^{(j)}\|_{k-j+3} 
\leq Nh^{k+1-j}(\|f\|_m+\|g\|_m).  
$$
Consequently,
\begin{equation*}
 \|F\|_{0}\leq N(\|f\|_m+\|g\|_m)h^{k+1}
\quad \text{for $h\in(0,h_0]$}, 
\end{equation*}
where  $N$ depends 
only   on $d$, $m$,   $ |\Lambda_1|  $, 
  $M_{ 0},\dots,M_m$, and $T$.
Hence we get \eqref{08.11.2.9} by 
Lemma \ref{lemma 3.6.1}, and the proof is   
complete. 
\end{proof}

\begin{proof}[Proof of Theorem \ref{theorem 1.25.11.08}]
Coming back to the above
  proof of Theorem \ref{theorem 6.25.08.08}
we see that function \eqref{08.12.16.3}
satisfies \eqref{6.25.08.08} with $F$ given by 
\eqref{08.12.16.4}. 
We notice that 
$$
\text{$k-j+3+l\leq m-3j$ \,for $j=0,1,\dots,k$
 \,in case (i)},  
$$
$$
\text{$k-j+3+l\leq m-2j$ \,for $j=0,1,\dots,k$ 
\,in case (ii)},  
$$
$$
\text{$k-j+3+l\leq m-2j$ \,for $j=0,1,\dots,k-1$ 
\,in case (iii)}.   
$$
Therefore 
by Theorem \ref{theorem 08.11.2.1}, 
when  $k=0$, and by Theorem \ref{theorem 08.11.7.2},    
when $k\geq1$, 
under each of (i), (ii), and (iii)
$$
\|u^{(j)}\|_{k-j+3+l}
\leq N(\|f\|_m+\|g\|_m)
$$
for $j=0,1\dots,k$ ($u^{(k)}=0$ in case (iii)). 
By Theorem \ref{theorem 08.11.3.1} the constant 
$N$ depends only on 
$m$, $\delta$, $c_0$, $\tau_0$, $K $, 
$M_{0}$, ..., 
$M_{m}$, $|\Lambda_{1}|$, and
$\|\Lambda_1\|$. 
By Lemma \ref{lemma 3.25.08.08}
$$
\|\mathcal O_h^{(k-j)}u^{(j)}\|_l\leq 
Nh^{k-j+1}\|u^{(j)}\|_{k-j+l+3}, 
$$
where $N$ is a constant depending only on 
 $|\Lambda_1|$, $M_0$,\dots $M_l$. 
Hence 
$$
\|F\|_l\leq N(\|f\|_m+\|g\|_m)h^{k+1}
\quad \text{for $h\in(0,h_0]$}. 
$$
Consequently, applying Theorem \ref{theorem 17.25.06.07} 
to equation \eqref{6.25.08.08},
for any multi-index $\alpha$, 
$|\alpha|\leq l$, for 
$$
r^{(\alpha)}_h:=h^{-(k+1)}
\big(D^{\alpha}u_h-\sum_{j=0}^{k}D^{\alpha}u^{(j)}
\frac{h^j}{j!}\big)
$$
we have 
$$
\|r^{(\alpha)}_h\|_{0}=
h^{-(k+1)}\|D^{\alpha}w\|_{0}
\leq N(\|f\|_m+\|g\|_m), 
$$
with a constant $N$ depending only on 
$m$, $d$, $\delta$, $c_0$, $\tau_0$, $K $, 
$M_{0}$, ..., 
$M_{m}$, $|\Lambda_{1}|$ and
$\|\Lambda_1\|$,   
which proves the theorem.
\end{proof}


\begin{thebibliography}{mm}

\bibitem{Bl} H. Blum, Q. Lin, and R. Rannacher, 
{\em Asymptotic Error Expansion and Richardson Extrapolation for 
linear finite elements\/}, Numer. Math., Vol. 49 (1986), 11-37. 

\bibitem{Br} C. Brezinski, {\em Convergence acceleration during 
the 20th century}, Journal of Computational and Applied 
Mathematics, Vol. 122 (2000), 1-21. 

\bibitem{DK} Hongjie Dong and 
N.V. Krylov,
{\em On the rate of convergence
of finite-difference approximations for
degenerate linear parabolic equations
with $C^{1}$ and $C^{2}$ coefficients\/},  
Electron. J. Diff. Eqns.,
Vol. 2005(2005), No. 102, pp. 1-25.
http://ejde.math.txstate.edu


\bibitem{J} D.C. Joyce, {\em Survey of extrapolation 
processes in numerical analysis}, SIAM Review, Vol. 13 (1971), 
No. 4, 435-490. 
 
\bibitem{GK1}I. Gy\"ongy and N.V. Krylov, 
{\em First derivative estimates 
for finite difference schemes\/}, 
Math. Comp., to appear,
http://arXiv.org/abs/0802.1180.



\bibitem{GK2}I. Gy\"ongy and N.V. Krylov, 
{\em Higher order derivative estimates for finite-difference 
schemes\/},  Methods and Applications of Analysis, to appear,
 http://arXiv.org/abs/0805.3149



\bibitem{Kr08} N.V. Krylov, {\em On
factorizations of smooth nonnegative matrix-values
functions and on
smooth functions with values in polyhedra\/},
 Appl. Math. Optim., Vol. 58 (2008), No. 3, 373-392. 

\bibitem{La73} O.A. Ladyzhenskaya, Boundary value problems of mathematical
physics,  Izdat. ``Nauka'', Moscow, 1973 (in Russian). 
 
\bibitem{Li} W. Littman,  {\em
R\'esolution du probl\`eme 
de Dirichlet par la m\'ethode des diff\'erences finies\/},
C. R. Acad. Sci. Paris, Vol. 247 (1958), 2270-2272. 

\bibitem{Ma} G.I. Marchuk,  
  Methods of numerical mathematics, Third edition,
``Nauka'', Moscow, 1989 (in Russian).

\bibitem{MS} G.I. Marschuk and V.V. Shaidurov, 
Difference methods and their extrapolations, New York 
Berlin Heidelberg, Springer 1983. 

\bibitem{Ol1} O.A. Ole$\rm \check i$nik, {\em 
 Alcuni risultati sulle equazioni
lineari e quasi lineari ellittico-paraboliche a derivate
parziali del secondo ordine\/}, (Italian) Atti Accad. Naz. Lincei
Rend. Cl. Sci. Fis. Mat. Natur., (8) 40, (1966), 775-784.

\bibitem{Ol} O. A. Ole$\rm \check i$nik, 
{\em On the smoothness of 
solutions of degenerating
elliptic and parabolic equations\/},   Dokl. Akad. Nauk
SSSR, Vol. 163 (1965),  577--580 in Russian; English translation
in Soviet Mat. Dokl., Vol. 6 (1965), No. 3, 972-976.

\bibitem{OR} O. A. Ole$\rm \check i$nik  and E. V. Radkevi$\rm \check c$,  
``Second order
equations with nonnegative characteristic form'',  
Mathematical analysis, 1969, pp. 7-252. (errata
insert) Akad. Nauk SSSR Vsesojuzn. Inst. Nau$\rm \check c$n. i Tehn.
Informacii, Moscow, 1971 in Russian; English translation:
Plenum Press, New York-London, 1973.



\bibitem{R} L.F. Richardson, {\em The approximative 
arithmetical solution by finite differences of physical problems 
involving differential equations}, Philos. Trans. Roy. Soc.
London, Ser. A, 210 (1910), 307-357. 

\bibitem{RG} L.F. Richardson and J.A. Gaunt, 
{\em The Deferred Approach 
to the Limit}, Phil. Trans. Roy. Soc. London Ser. A, Vol. 226 (1927),  
299-361.

\end{thebibliography}
\end{document}